\documentclass[bj,preprint]{imsart}

\RequirePackage[OT1]{fontenc}
\RequirePackage{amsthm,amsmath}
\RequirePackage[sort,numbers]{natbib}
\RequirePackage[colorlinks,citecolor=blue,urlcolor=blue]{hyperref}

\usepackage[utf8x]{inputenc}
\usepackage{mathrsfs}
\usepackage{amssymb}
\usepackage{amsfonts}
\usepackage{wrapfig}
\usepackage{multirow}
\usepackage{longtable}
\usepackage{enumerate,comment}
\usepackage{color, graphicx, enumitem}
\usepackage{tikz}

  \newlength\squareheight
  \newcommand\squareslash{\tikz{\draw (0,0) rectangle (\squareheight,\squareheight);\draw(0,0) -- (\squareheight,\squareheight)}}
  
  \newcommand\squarebslash{\tikz{\draw (0,0) rectangle (\squareheight,\squareheight);\draw(0,\squareheight) -- (\squareheight,0)}}
  
  \newlength\hsquareheight
  \newcommand\squarehslash{\tikz{\draw (0,0) rectangle (\squareheight,\squareheight);\draw(0,\hsquareheight) -- (\squareheight,\hsquareheight)}}

\arxiv{arXiv:0000.0000}

\DeclareFontFamily{U}{mathx}{\hyphenchar\font45}
\DeclareFontShape{U}{mathx}{m}{n}{
      <5> <6> <7> <8> <9> <10>
      <10.95> <12> <14.4> <17.28> <20.74> <24.88>
      mathx10
      }{}
\DeclareSymbolFont{mathx}{U}{mathx}{m}{n}
\DeclareFontSubstitution{U}{mathx}{m}{n}
\DeclareMathAccent{\widecheck}{0}{mathx}{"71}
\DeclareMathAccent{\wideparen}{0}{mathx}{"75}

\newcommand{\BR}{\mathbb{R}}

\newcommand{\mB}{{\cal B}}

\newcommand{\bbeta}{{\boldsymbol \beta}}

\newcommand{\bvarepsilon}{{\boldsymbol \varepsilon}}

\def\boxit#1{\vbox{\hrule\hbox{\vrule\kern6pt
          \vbox{\kern6pt#1\kern6pt}\kern6pt\vrule}\hrule}}

\startlocaldefs
\numberwithin{equation}{section}
\theoremstyle{plain}
\newtheorem{theorem}{Theorem}[section]
\newtheorem{lemma}{Lemma}[section]

\newtheorem{corollary}{Corollary}[section]

\newtheorem{remark}{Remark}
\endlocaldefs

\begin{document}

\begin{frontmatter}
\title{Optimal False Discovery Control of Minimax Estimators}
\runtitle{Optimal False Discovery Control}

\begin{aug}
\author{\fnms{Qifan} \snm{Song}\thanksref{t2}\ead[label=e1]{qfsong@purdue.edu}}
\and
\author{\fnms{Guang} \snm{Cheng}\thanksref{t2}\ead[label=e2]{guangcheng@ucla.edu}}

\thankstext{t1}{Sponsored by NSF DMS-1811812}
\thankstext{t2}{Sponsored by NSF DMS-1712907, DMS-1811812, DMS-1821183, NSF-SCALE MoDL (2134209)}
\runauthor{Q. Song and G. Cheng}

\affiliation{Purdue University}


\end{aug}

\begin{abstract}
Two major research tasks lie at the heart of high dimensional data analysis: accurate parameter estimation and correct support recovery. The existing literature mostly aims for either the best parameter estimation or the best model selection result, however little has been done to understand the potential interaction between the estimation precision and the selection behavior. In this work, our minimax result shows that an estimator's performance of type I error control directly links with its $L_2$ estimation error rate, and reveals a trade-off phenomenon between the rate of convergence and the false discovery control: to achieve better accuracy, one risks yielding more false discoveries. In particular, we characterize the false discovery control behavior of rate optimal and rate suboptimal estimators under different sparsity regimes, and discover a rigid dichotomy between these two estimators under near-linear and linear sparsity settings. In addition, this work provides a rigorous explanation to the incompatibility phenomenon between selection consistency and rate minimaxity which has been frequently observed in the high dimensional literature.
\end{abstract}



\end{frontmatter}

\section{Introduction}\label{intro}


Modern studies in the sciences collect huge data sets which include information of a large number of potential explanatory variables, and then attempt to discover the possible association between these variables and the response of interest. For example,
in the genome-wide association study (GWAS), where researchers want to find which genetic variants are associated with a trait,
we collect high dimensional single-nucleotide polymorphism (SNP) arrays and then 
aim at finding the association between the trait and SNPs.


For simplicity, we start with the normal means model, which is the simplest form of high dimensional regression:
\begin{equation}\label{normalmean}
    y=\beta+\varepsilon,
\end{equation}
where $y\in\BR^n$, $\varepsilon\sim N(0,I_n)$, and $\beta$ is a $p_n$-dimensional regression parameter with $p_n=n$.
The parameter $\beta$ is assumed to be sparse and has at most $s_n$ nonzero entries. Throughout this paper, we assume $\limsup s_n/p_n = \zeta$, where the constant $\zeta\in[0,1)$. Note this sparsity setting includes both the classical strict sparse model, i.e., $s_n/p_n\rightarrow 0$, and the linear sparsity model, i.e., $\lim s_n/p_n\in(0,1)$ (e.g.,\cite{SuBC2017}).

Given an estimator $\widehat\beta$, there are generally two types of evaluation on its performance: one is the quantitative accuracy which is usually measured by $L_2$ distance $\|\widehat\beta-\beta\|$; another is the selection correctness, i.e., the difference between the two sparse models induced by $\widehat\beta$ and $\beta$. In the past decades, there is rich literature studying either estimation accuracy (e.g., rate of convergence, minimax estimation) or selection behavior (e.g., false discovery rate control, variable selection consistency, Hamming loss). But few research has been conducted to understand the interplay between estimation accuracy and selection behavior. It is not only of mathematical interest to investigate the relationship between estimation and selection. One important practical question would be ``can we find one optimal estimation for the purpose of both accuracy estimation and correct selection?''. The obtained statistical insights would also be beneficial to applied researchers who seek decision making on the goodness of both estimation and selection in a certain way, e.g., when the risk is defined as some combination of estimation error and selection error. 

The first objective of this work is to investigate the relationship between $L_2$ convergence rate and false discovery, or type I error, control. Variable selection of $\beta$ can be alternatively viewed as a multiple hypotheses testing problem $$H_j: \beta_j=0\,\;\;\mbox{for all}\;\;j=1,\dots,p_n.$$  In the literature of multiple testing problems, one usually aims to control and balance type I and type II errors, such that one can make as many rejections as possible while type I errors are still under control, in the sense of, e.g., bounded familywise error rate or bounded false discovery rate (FDR \cite{BenjaminiH1995}).
For instance, \cite{arias2017distribution,bogdan2011asymptotic,salomond2017risk} explored the optimality of independent multiple testing procedure in terms of minimizing the sum of false positive (rate) and false negative (rate) under certain asymptotic settings.  
These aforementioned works help us to design inference procedures that pursue the greatest testing power subject to limited type I errors. However, for researchers  who care about both prediction accuracy and model interpretability, a different task may arise, that is to pursue the smallest estimation error subject to limited type I errors.

The most relevant work in the literature that established connections between the convergence rate and false discovery control is done by \cite{AbramovichBDJ2006}. It studied the convergence of the Benjamin-Hochberg (B-H) estimator $\widehat \beta(y,q_n)$, where $\widehat\beta_i = y_i$ if $H_i$ is rejected by  Benjamin-Hochberg FDR control procedure \cite{BenjaminiH1995} with targeted FDR level $q_n$. \cite{AbramovichBDJ2006} proved that if $s_n\in[\log^5n, n^{1-\delta}]$, $q_n\geq \gamma/\log n$ and $\lim q_n=q\leq 1/2$ for some $\delta,\gamma>0$, then the B-H estimator is sharply minimax, i.e.,
\[\begin{split}
&\lim_{n\rightarrow\infty}\frac{\sup_{\{\beta \mbox{ is $s_n$-sparse}\}}E_\beta\|\beta-\widehat\beta(y,q_n)\|^2}{R_{opt}}= 1,\mbox{ and }\\
&R_{opt}=\inf_{\widehat\beta} \sup_{\{\beta \mbox{ is $s_n$-sparse}\}}E_\beta\|\beta-\widehat\beta\|^2
=(2+o(1))s_n\log(n/s_n),
\end{split}
\]
where the ``inf'' is taken over all estimators, the ``sup'' is taken over all possible true parameters $\beta$'s with at most $s_n$ nonzero entries, and the little-o notation $o(1)$ stands for some sequence converging to 0. In other words, we indeed have a good choice of estimator $\widehat \beta(y,\gamma/\log n)$, which guarantees a decreasing FDR and is also optimally accurate. Similar results that show a connection between the FDR control and rate minimaxity can also be found in \cite{castillo2018spike} and \cite{SuC2016}.

Beyond these existing results, this paper aims to provide more general in-depth insight of the interaction between accuracy and false discovery control. We investigate this matter by figuring out the best achievable (i.e., minimax) type I error control behavior subject to a pre-specified rate of convergence. The quality of type I error control can be naturally measured by the number of false positive selections, and our main theorem
establishes the following ``minimax-type'' result for normal means problem:
\begin{eqnarray}\label{res00}
\inf_{\{\widehat\beta\;\mbox{is rate-$\sqrt{s_nf_n}$}\}}\sup_{\{\beta \mbox{ is $s_n$-sparse}\}}E[\mbox{FP}(\widehat\beta)]\asymp p_n\exp\{-(0.5+o(1))f_n\},\label{mini}
\end{eqnarray}
where $\mbox{FP}(\widehat\beta)$ denotes the number of false positives resulting from $\widehat\beta$, the ``inf'' is taken over all estimators whose $L_2$ convergence rates are no larger than $\sqrt{s_nf_n/n}$ for some given function $f_n$, which may depend on $n, p_n$ and $s_n$, and ``$\asymp$'' represents the equivalency in asymptotic order (refer to the Notation paragraph at the end of Section \ref{intro}). 
The above result (\ref{mini}) essentially reveals a trade-off between estimation accuracy and false discovery control, that is, a more accurate estimator (i.e., a smaller $f_n$) can yield more false discoveries and vice versa. 
The rationale behind this counter-intuitive result is that, a good estimation requires a delicate balance between type I and type II errors. Aggressively avoiding type I errors will inevitably increase the number of type II errors, and lead to an inaccurate estimation; vice versa, an accurate estimation must tolerate the occurrence of type I errors.

Note that the above minimax result can be re-phrased as follows: the best possible convergence rate of $\widehat\beta$, subject to that its average number of false discoveries is always bound by $d$, is
\[
\sup_{\beta\mbox{ is $s_n$-sparse}}E\|\widehat\beta-\beta\|^2\asymp {s_n}\log\left(\frac{p_n}{d}\right).
\]


In the literature, various estimators achieve rate optimal convergence for the normal means sequence model, i.e., their $L_2$ convergence rate is of the order $\sqrt{s_n\log(p_n/s_n)}$ \cite{Castillov2012, AbramovichBDJ2006, WuZ2013, BhattacharyaPPD2015, Rockova2015,SilvermanJ2004, song2020bayesian}.
For such a class of rate minimax estimators, our general result (\ref{mini}) implies that 
\begin{equation}\label{res0}
\begin{split}
\inf_{\{\mbox{rate-minimax } \widehat\beta \}}\sup_{\{\mbox{$s_n$-sparse }\beta \}}&E(\mbox{FP})\asymp p_n(s_n/p_n)^{\tau+o(1)},
\end{split}
\end{equation}
and we furthermore show that 
\begin{equation}\label{res1}
\begin{split}
\inf_{\{\mbox{rate-minimax } \widehat\beta \}}\sup_{\{\mbox{$s_n$-sparse }\beta \}}&\mbox{FDR}\asymp (s_n/p_n)^{\tau-1+o(1)}.
\end{split}
\end{equation}
In above equations (\ref{res0}) and (\ref{res1}), the ``inf'' is taken over all estimators $\widehat\beta$ whose $L_2$ convergence rate are smaller than $(2\tau s_n\log(p_n/s_n))^{1/2}$ for some constant $\tau$, i.e., $2\tau$ is the multiplicative constant of the convergence rate of $\widehat\beta$. Therefore, both minimax false discovery number and minimax false discovery rate follow a polynomial of the sparsity ratio. It is worth mentioning that \eqref{res1} requires that $\tau>1$ strictly, which excludes the sharply minimax estimators (e.g., the B-H estimator). Therefore, our contribution on characterizing the minimax FDR of rate minimax estimator is not an improvement over existing work \cite{AbramovichBDJ2006}, but a valuable addition.

A toy simulation is conducted to verify the polynomial convergence speed for the false discovery rate in the normal means models where $n=p_n$ and $s_n=n^{1/2}$ and all nonzero $\beta_i$'s are set as $[2\log(n/s)]^{1/2}$. Note that this choice of magnitude of $\bbeta$ represents the most difficult case for minimax estimators to control false discovery (refer to our discussion for Theorem \ref{impossibility3}). The rate minimax estimator (\ref{SCest}) in Section \ref{fest} is used in the simulation with tuning parameter value $\gamma=2.1$. Figure \ref{fdpplot} plots the logarithm of estimated FDR based on 100 independent simulations versus the logarithm of 
true sparsity ratio $\log(s_n/n)$. The plot displays a clear and strong linear trend with $R^2=0.9882$.

\begin{figure}[htp]
 \begin{center}
  \includegraphics[width=9cm]{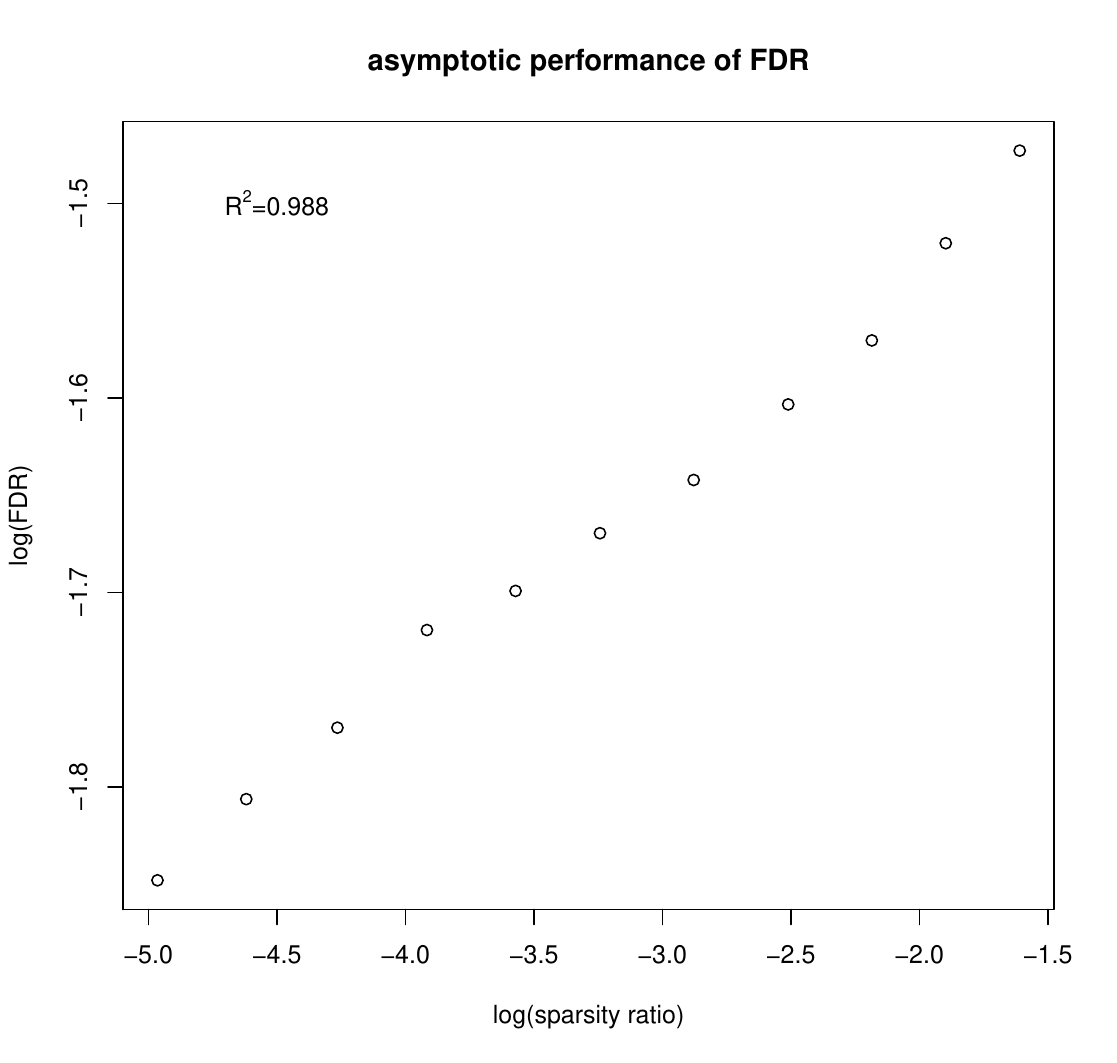}\label{fdpplot}\vskip -0.25in
  \caption{This figure plots the patter of FDR with respect to the sparsity ratio. }
 \end{center}
\end{figure}

The analysis carried out in this paper enables us to characterize the best possible type I error control behavior of rate optimal estimators (with suboptimal multiplicative constant) according to the sparsity growth rate: 
\begin{enumerate}
	\item (Polynomial sparsity) If $s_n\leq p_n^{\alpha}$ for some $\alpha<1$, then rate minimax estimators, at best, guarantee that there is no false discovery asymptotically (as long as its multiplicative constant of $L_2$ convergence is sufficiently large);
	
	\item (Near-linear sparsity) If $s_n/p_n\rightarrow 0$ and $\log s_n / \log p_n\rightarrow 1$,  rate minimax estimators, at best, guarantee that the false discovery rate decays to 0 at a polynomial rate of $s_n/p_n$. However, under the worst-$\beta$ scenario, its expected number of false positive selections diverges to infinite and its probability of committing false positives converges (i.e., the familywise error rate) to 1.
	
	\item (Linear sparsity) If $s_n=\zeta p_n$ for some fixed $\zeta\in(0,1)$, for any rate minimax estimator, its false discovery rate is bound away from 0 and its number of false discoveries tends to infinity, under the worst-$\beta$ scenario.
\end{enumerate}

As revealed by \eqref{res0} and \eqref{res1}, the multiplicative constant matters. Under polynomial sparsity, rate minimax estimators can guarantee no false discovery only when the multiplicative constant scales up with $\alpha$, and sharply minimax estimators (whose multiplicative constant is 2) never guarantee zero false discovery no matter how small the $\alpha$ is  (refer to Figure \ref{smpplot}). But under near-linear sparsity, rate minimaxity estimators may always yield false discoveries regardless of the multiplicative constant, and this finding is consistent with existing results, e.g., \cite{BhattacharyaPPD2015, GaovZ2015} showed that some Bayesian rate minimax estimation procedure selects larger models than the true model. 
It is worth mentioning that the near-linear sparsity setting, e.g., $s_n=p_n/\log p_n$, is still a strict sparse setting in the sense that the sparsity ratio tends to 0. 

In the literature, another common $L_2$ convergence rate for $\widehat\beta$ is $\sqrt{s_n\log p_n}$, which is usually attained by model selection consistent estimators (e.g., soft or hard thresholding estimators with threshold $\sqrt{2\log p_n}$). This rate is referred as ``suboptimal rate'', since it is inferior to minimax rate under near-linear or linear sparsity. Interestingly, we show that rate suboptimal estimators, regardless of the sparsity growth, are always capable to achieve zero false discovery asymptotically (in mean or in probability), providing that the multiplicative constant is greater than 2.


To compare the class of rate optimal and the class of rate suboptimal estimators: 1) Under polynomial sparsity, two actually share the same rate but with potentially different multiplicative constants. Hence they can achieve that same best possible type I error control limit, as long as one is willing to compromise on larger multiplicative constants. For example, the tuning parameter in a rate minimax estimation procedure needs to be properly selected, such that the yielded multiplicative constant is adaptive to the sparsity growth. 2) Beyond polynomial sparsity,
the two become totally different in terms of balancing estimation precision and type I error. Specifically, the suboptimal estimators attain the best possible estimation accuracy subject to zero false discovery, while rate optimal estimators can achieve strictly better estimation accuracy at the expense of worse type I error control performance. In other words, under relatively denser sparsity regimes, rate optimal and rate suboptimal estimators represent two different balancing strategies. These asymptotic characterizations don't directly apply to real applications, since one can not determine the sparsity regime of a finite-sample data set. However, the insights behind these theoretical results suggest that: for high dimensional problems where the true models are likely to be very dense, researchers shall carefully choose rate optimal or rate suboptimal estimators that best fits their research objectives. As an instance of large true models, in Gene Regulatory Networks studies, it is usually believed that there are thousands of active genes interacting with each others (e.g., \cite{wilkins2016egrins} selected 5,447 genes among 56,000 genes, \cite{koda2017diurnal} selected 2,187  genes among 3,621 genes).


Our study of the relationship between $L_2$ convergence and false discovery control is further generalized to the high dimensional linear models:
\begin{equation}\label{normalregression}
 y=X\beta+\bvarepsilon,
\end{equation}
where $X\in\BR^{n\times p_n}$, the dimension $p_n$ is potentially much larger than the number of observations $n$, and $\beta$ is an $s_n$-sparse parameter. In the literature, one usually performs penalized estimation:
\[
\widehat\beta = \underset{\beta}{\mathrm{argmin}} \|y-X\beta\|^2+pe(\beta),
\]
for some penalty function $pe$. Two typical examples are LASSO \cite{Tibshirani1996} ($pe(\beta)=\lambda \sum_{j=1}^{p_n} |\beta_j|$) and SLOPE \cite{BogdanvSSC2015} ($pe(\beta)=\lambda\sum_{j=1}^{p_n}\Phi^{-1}(1-jq/2p_n) |\beta|_{(j)}$\footnote{$|\beta|_{(j)}$ is the  sorted descending value of $|\beta_j|'s$.}), where the former is not adaptive to unknown sparsity (i.e., it requires the knowledge of true $s_n$ to attain minimax rate), and the latter one is adaptive to unknown sparsity.
Compared with the normal means models \eqref{normalmean}, there is usually a root-$n$ difference between the convergence rates for the two models, i.e., the minimax rate for regression models is $\sqrt{s_n\log(p_n/s_n)/n}$ \cite{RaskuttiWY2011} under proper regularity conditions. The column dependencies of design matrix $X$ complicate the analysis of false discovery control, and there are limited studies characterizing the selection behavior for regression models: for instance, \cite{SuBC2017, WeinsteinBC2017} studied the relationship between type I and type II errors along the LASSO regularization path. In terms of the relationship between accuracy and false discovery control, to the best of our knowledge, the only significant result is that the SLOPE estimator achieves almost sharply minimax and ensures a decaying FDR \cite{SuC2016,KosB2019} under Gaussian design. Our paper shows that, under Gaussian design and proper regularity of dimensionality, results \eqref{res00}, \eqref{res0} and \eqref{res1} still hold, except for a root-$n$ adjust on the convergence rate.

The second objective of this work is to understand the relationship between convergence rate and model selection consistency. This particular research objective is motivated by the observation that in the literature, most regression estimators can not achieve rate optimality and selection consistency simultaneously, even under strict sparsity assumption that $s_n/p_n\rightarrow 0$. For example, under normal means models with known error variance $\sigma=1$, the hard thresholding estimator is defined as $\widehat\beta_i = y_i1(|y_i|>t)$. When $t\geq(2\log(n-s_n))^{1/2}$, it achieves selection consistency, but only attains suboptimal convergence rate. When $t=\sqrt{c\log(n/s_n)}$ for some constant $c$, it obtains rate minimax convergence \cite{DonohoJ1994}, but will yields inconsistent selection under near-linear sparsity. Similar phenomenon occurs to the LASSO estimator as well. When the tuning parameter of LASSO is of an order $\sqrt{\log p_n/n}$, LASSO leads to selection consistency but rate suboptimal estimation; when the tuning parameter is of a smaller order $\sqrt{\log (p_n/s_n)/n}$, LASSO estimator is rate optimal \cite{YeZ2010}, but inconsistent for variable selection under near-linear sparsity.

One must be aware of that, false discovery control and model selection consistency, although both measure the goodness of selection procedures, consider different domains of parameter space. The former is interested in all $s_n$-sparse $\beta$'s regardless of the magnitudes of its nonzero entries, while the latter is only interested in the $s_n$-sparse $\beta$’s whose nonzero entries are greater than a certain minimal threshold ({depending on $n$, $s_n$ and $p_n$}), i.e., the so-called the beta-min condition in the literature. 
Hence, selection consistency is not a sufficient condition for zero false discovery. Our conclusion that rate optimal estimators cannot guarantee zero false discovery {uniformly for different sparsity levels (as long as $s_n=o(p_n)$)}, doesn't directly imply that, rate optimal estimators cannot guarantee model selection consistency uniformly for different sparsity levels.

Our theoretical investigation reveals a general rule that for normal means models \eqref{normalmean}: if an estimator possesses a certain monotone property (to be defined later), then this estimator can never be both rate minimax and selection consistent uniformly over all sparsity regimes. To be specific, this incompatibility occurs under near-linear and linear sparsity.
This explains that in the literature, researchers fail to prove selection consistency for rate optimal estimators, since almost all the penalized estimators or Bayesian estimators are monotone. Therefore, the best possible selection consistency result for monotone rate optimal estimators is categorized as 

(a) under polynomial sparsity, they can achieve selection consistency; 

(b) under near-linear sparsity, they can achieve almost full model recovery \cite{ButuceaST2017}, that is, 
\[ \frac{\mbox{the number of false positives and false negatives}}{\mbox{true model size}}\longrightarrow 0 \mbox{ in probability}.\]



The above two research objectives, together, precisely depict the interplay and trade-off between estimation convergence speed and selection correctness. 
Compared with the aforementioned existing results,
which studied specific estimators (i.e., B-H or SLOPE estimator) within a narrower range of sparsity (e.g., polynomial sparsity), our work has a more general scope that aims to understand the minimax false discovery performance of estimators with different convergence speed, under various sparsity scenarios. In addition, we make a subtle comparison of the false discovery control behavior between rate optimal estimators and rate suboptimal estimators. Figure \ref{smpplot} in Section \ref{impo} summarizes what is the best possible false discovery control performance of rate optimal and suboptimal estimators under different sparse regimes for the normal means models, and delivers the general insight that for rate minimax estimators, the difficulty of false discovery control increases as sparsity ratio increases. 


This paper is organized as follows.
In Section \ref{impo}, we investigate the relationship between the rate of convergence and false discovery rate, and establish the minimax rate for the number of false discoveries and the false discovery rate, for both normal mean models and regression models with Gaussian random design.
In Section \ref{impo2}, we establish the incompatibility between rate minimaxity and model selection consistency for monotone estimators under normal means models.
Section \ref{fest} shows that the minimax rate discussed in Sections \ref{impo} can be achieved 
by adaptive estimators induced by $\sum_{i=1}^k\log(p/i)$-penalization. Some more discussion and remarks are provided in Section \ref{diss}.
All the proofs are provided in the supplementary material.

\vskip 0.2in
{\noindent\bf Notation.} We use $\xi\subset\{1,\dots,p\}$ to denote a model subset, and let $|\xi|$ be the size of this model.
For any vector $\beta\in\BR^{p}$ and design matrix $X\in\BR^{n\times p}$,
$\beta_\xi$ and $X_\xi$ denote the sub-vector or submatrix corresponding to the model $\xi$. With slight abuse of the notation, we use $\xi(\beta)$ as the operator that extracts the model of $\beta$, i.e.,
$\xi(\beta)=\{j; \beta_j\neq 0\}$. Denote $\|\beta\|_{\min}=\min_{\beta_i\neq0}|\beta_i|$, i.e., the minimum signal strength of $\beta$. 
Let $\mB(p_n,s_n)=\{\beta\in\BR^{p_n}: \|\beta\|_0\leq s_n\}$ denote all $s_n$-sparse vectors in the $p_n$ dimensional space.
For two sequences of positive values $\{a_n\}_{n=1}^\infty$ and $\{b_n\}_{n=1}^\infty$, $a_n\succ b_n$ means that $\lim_n a_n/b_n = \infty$,
$a_n\asymp b_n$ means $0<\lim\inf a_n/b_n\leq\lim\sup a_n/b_n<\infty$, 
$a_n\sim b_n$ means that $\lim a_n/b_n =1$, and 
$a_n\lesssim b_n$ means that $\limsup(a_n-b_n)\leq0$.
Little o notation $o(1)$ represents some sequence converging to 0, and $o^+(1)$ represents some positive sequence converging to 0.

\section{Estimation Rate and False Discovery Control}\label{impo}

In this section, we investigate the trade-off between type I error control behavior and rate of convergence under both normal means models and  linear regression models. Specifically, we characterize false discovery control behaviors (including the number of false discoveries and false discovery rate) of (sub-)optimal estimators under different sparsity regimes. A general message is that an estimator can yield more type I errors if it possesses a faster convergence rate. 

\subsection{Normal Means Models}\label{nmm}
We first illustrate our main results under normal means models
$y=\beta+\varepsilon,$ where
$\varepsilon\sim N(0,I_n)$ and the true parameter $\beta\in\mB(n,s_n)$.
An estimator, denoted by $\widehat \beta(\cdot)\in \mathbb{R}^n$, is any function of $y$. It can be either an adaptive estimator, which does not require any knowledge about the unknown sparsity $s_n$, or a non-adaptive one which involves the unknown $s_n$.  
To state our result, we denote $\mbox{FP}(\widehat\beta)$ as the number of false discoveries yielded by any estimator $\widehat\beta(y)$.
Given a function $f(n,s_n)$, let $$\Omega_f=\{\widehat\beta: \sup_{\beta\in\mB(n,s_n)}E_{\beta}\|\beta-\widehat \beta\|^2\leq s_nf(n,s_n)\}$$ be the collection of estimators with convergence rate $\{s_nf(n,s_n)\}^{1/2}$, where $E_{\beta}$ denotes the expectation with respect to the data $y$ that is generated under true parameter $\beta$. Note that the function $f(n,s_n)$ must be bounded away from 0, since the oracle parametric rate is $O(\sqrt{s_n})$. 

Our first theorem studies the minimax lower bound for the expected number of false discoveries.
\begin{theorem}\label{impossibility3}
For the normal means model, the following non-asymptotic lower bound holds:
\begin{equation}\label{nonasy}
\inf_{\widehat\beta\in\Omega_f}\sup_{\beta\in\mB(n,s_n)}E_{\beta} [{\rm FP}(\widehat\beta)] \geq (n-s_n+1)\sup_{c>1}\Phi(\Phi^{-1}(1-1/c)-\{cf(n,s_n)\}^{1/2}),
\end{equation}
where $\Phi(x)$ is the distribution function of standard normal distribution.
Asymptotically, if $\lim\sup f(n,s_n)=\kappa<\infty$, then as $n$ tends to infinity,
\[\inf_{\widehat\beta\in\Omega_f}\sup_{\beta\in\mB(n,s_n)}E_{\beta} [{\rm FP}(\widehat\beta)] \geq H(\kappa)\times(1-\zeta)n\]
where $\zeta=\lim s_n/p_n$ and $H(\kappa)\in(0,1)$ is a positive constant that decreases as $\kappa$ increases (the definition of function $H(\cdot)$ can be found in the proof in the supplementary material);
if $\lim f(n,s_n)=\infty$, then 
\[\inf_{\widehat\beta\in\Omega_f}\sup_{\beta\in\mB(n,s_n)}E_{\beta} [{\rm FP}(\widehat\beta)] \geq C(1-\zeta)n\exp\{-[1+o^+(1)]f(n,s_n)/2\},\]
where $C$ is some absolute constant, and $o^+(1)$ represents some positive sequence converging to 0.
\end{theorem}

To complete the minimax analysis, an upper bound result is necessary. Since $\mbox{FP}(\widehat\beta)\leq n$, thus trivially 
$$\inf_{\widehat\beta\in\Omega_f}\sup_{\beta\in\mB(n,s_n)}E_{\beta} [{\rm FP}(\widehat\beta)] = O(n),$$ if $f(n,s_n)\asymp 1$.
On the other hand, if $f(n,s_n)\rightarrow\infty$, we let
$\lim\inf s_nf(n,s_n)/r_{\rm opt}\geq1$ which ensures that $\Omega_f$ is not empty, where $r_{\rm opt}=\inf_{\widehat\beta}\sup_{\beta\in\mB(n,s_n)}E_\beta\|\widehat\beta-\beta\|^2$ denotes the minimax $L_2$ rate\footnote{Here, $r_{\rm opt}$ represents the minimax rate when $\lim s_n/n<\zeta\in[0,1)$, while $R_{\rm opt}$ in the introduction section represents the minimax rate when $\lim s_n/n=0$}. 
Note that this also implies that $\lim\inf [f(n,s_n)/\log(n/s_n)]\geq 2$\footnote{This is because, if $s_n=o(n)$, $r_{\rm opt}=(2+o(1))s_n\log(n/s_n)$; if not, $f(n,s_n)\geq 2\log(n/s_n)$ holds trivially since $\log(n/s_n)$ is bounded.}.

Given such a diverging $f$, we can define a hard-thresholding estimator $$\widehat\beta_i=y_i1\{y_i^2\geq (1-\epsilon_n) f(n,s_n)\}$$ for some $1/f(n,s_n)\prec\epsilon_n\prec 1$. One can show that asymptotically, it belongs to $\Omega_f$, since
\begin{equation}\label{ratequ0}
\begin{split}
 \sup_{\beta\in\mB(n,s_n)}E\|\widehat\beta-\beta\|^2
\leq O(s_n)+ s_n + (1-\epsilon_n)s_nf(n,s_n)< s_nf(n,s_n),
\end{split}
\end{equation}
as $n\rightarrow \infty$. The proof of \eqref{ratequ0} can be found in the supplementary material. And trivially, 
$\mbox{FP}(\widehat\beta)\sim\mbox{Bin}(n-\|\beta\|_0, 2\Phi(-\{(1-\epsilon_n)f(n,s_n)\}^{1/2}))$, and 
\[
E[\mbox{FP}(\widehat\beta)]=2(n-\|\beta\|_0)\Phi(-\sqrt{ (1-\epsilon_n)f(n,s_n)}).
\]
Since $\Phi(x)=O(\exp(-x^2/2))$, we have
\begin{theorem}\label{upper1}
For the normal means model, the following asymptotic upper bound holds:
If $\lim\sup f(n,s_n)=\kappa<\infty$, then as $n$ tends to infinity,
\[\inf_{\widehat\beta\in\Omega_f}\sup_{\beta\in\mB(n,s_n)}E_{\beta} [{\rm FP}(\widehat\beta)] = O(n).\]
If $\lim f(n,s_n)=\infty$, then 
\[\inf_{\widehat\beta\in\Omega_f}\sup_{\beta\in\mB(n,s_n)}E_{\beta} [{\rm FP}(\widehat\beta)] \leq Cn\exp\{-[1-o^+(1)]f(n,s_n)/2\},\]
where $C$ is some absolute constant.
\end{theorem}

The asymptotic lower bound (Theorem \ref{impossibility3}) and upper bound (Theorem \ref{upper1})
match except for an $o^+(1)$ term difference. Therefore, uniformly for both diverging and non-diverging $f$ cases, we claim the following minimax false positive control result:
\begin{equation}\label{minimax0}
\inf_{\widehat\beta\in\Omega_f}\sup_{\beta\in\mB(n,s_n)}E_{\beta} [{\rm FP}(\widehat\beta)] \asymp n\exp\{-[1+o(1)]f(n,s_n)/2\}, \mbox{ as } n\rightarrow \infty.
\end{equation}

The above analysis reveals a trade-off between estimation accuracy and false discovery control. Specifically, if $f_1\gg f_2$ in the sense that $\lim_n f_1-f_2=\infty$, then we have $$\inf_{\widehat\beta\in\Omega_{f_1}}\sup_{\beta\in\mB(n,s_n)}E_{\beta} [{\rm FP}(\widehat\beta)]\prec \inf_{\widehat\beta\in\Omega_{f_2}}\sup_{\beta\in\mB(n,s_n)}E_{\beta} [{\rm FP}(\widehat\beta)].$$ Note that the hard-thresholding estimate, which is used to establish the upper bound in Theorem \ref{upper1}, relies on the value of $s_n$ if $f(n,s_n)$ depends on $s_n$. Later in Section \ref{fest}, we will discuss some adaptive estimator that does not depend on $s_n$.

\begin{remark}
It is important to remind the readers that the $o(1)$ term appearing in the exponent term of (\ref{minimax0}) is not always negligible. When $f$ is diverging, this $o(1)$ term can affect the order, i.e., $\inf_{\widehat\beta\in\Omega_f}\sup_{\beta\in\mB(n,s_n)}E_{\beta} [{\rm FP}(\widehat\beta)]$ may have strictly greater or smaller order than $n\exp\{-f(n,s_n)/2\}$. An equivalent representation of \eqref{minimax0} is 
\[\log\left(\inf_{\widehat\beta\in\Omega_f}\sup_{\beta\in\mB(n,s_n)}E_{\beta} [{\rm FP}(\widehat\beta)]\right) \sim \log n-[1+o(1)]f(n,s_n)/2.\]
Note that the above argument applies to our other minimax results, e.g., (\ref{minimax1}) (\ref{minimaxfdrs}) and (\ref{minimaxfdrlinear}), as well.
\end{remark}

\begin{remark} \label{re1}The right-hand side of the inequality (\ref{nonasy}) is attained when all nonzero entries of the true $\beta$ are of magnitude $\sqrt{f(n,s_n)}$ (Refer to the construction of prior $\pi$ in the proof of Theorem \ref{impossibility3} in the supplementary material). This gives the most difficult case of false discovery control for estimators in $\Omega_f$.
\end{remark}

Two choices of $f(n,s_n)$'s are of interest. The first one is $f(n,s_n)=c_1\log(n/s_n)$ which represents the minimax $L_2$ convergence rate with a multiplicative constant $c_1\geq2$. Respectively, we define
$$\Omega_o(c_1)=\{\widehat\beta: \sup_{\beta\in\mB(n,s_n)}E_{\beta}\|\beta-\widehat \beta\|^2\leq c_1 s_n\log(n/s_n)\}.$$

The next result is a direct consequence of Theorems \ref{impossibility3} and \ref{upper1}.

\begin{corollary}\label{impossibility3-2}
For the normal means model, the following asymptotic result holds for the class of rate optimal estimators 
 with $c_1\geq2$,
\begin{equation}\label{minimax1}
\begin{split}
&\inf_{\widehat\beta\in\Omega_o(c_1)}\sup_{\beta\in\mB(n,s_n)}E_{\beta} [{\rm FP}(\widehat\beta)] \asymp n,\quad\mbox{if } \lim(s_n/n)>0;\\
&\inf_{\widehat\beta\in\Omega_o(c_1)}\sup_{\beta\in\mB(n,s_n)}E_{\beta} [{\rm FP}(\widehat\beta)] \asymp
n(s_n/n)^{(c_1+o(1))/2},\quad\mbox{if } \lim (s_n/n)=0.
\end{split}
\end{equation}
\end{corollary}

By similar arguments of Theorem \ref{upper1}, the minimax rate in Corollary \ref{impossibility3-2} can be attained by some non-adaptive hard-thresholding estimator $$\widehat\beta_i(y) = y_i1(|y_i|^2\geq (c_1-o^+(1))\log(n/s_n)),$$
and an adaptive estimator that achieves this minimax rate in probability is discussed in Section \ref{fest}.
Under $s_n/n\rightarrow 0$, $c_1$, the multiplicative constant of the $L_2$ convergence rate, does negatively affect the minimax rate of false discoveries,
as it appears in the exponent term of $(s_n/n)$. However, under linear sparsity, 
different values of $c_1$ do not change the order in (\ref{minimax1}). Hence, it is not clear how $\inf_{\widehat\beta\in\Omega_o(c_1)} \sup_{\beta\in\mB(n,s_n)}E_\beta(\mbox{FP}(\widehat\beta))$ changes with respect to $c_1$.
We believe that the trade-off phenomenon between accuracy and false discovery control applies to the case of linear sparsity as well, and conjecture that the multiplicative constant negatively impacts false discovery control as well.

\begin{remark}
It is well known in the literature that under strict sparsity, $\beta_j^2\approx c_1\log(n/s_n)$ for all nonzero entries is the most difficult case to be accurately estimated for rate minimax estimators. Remark \ref{re1}  implies that this is also the worst-$\beta$ case for rate minimax estimators to control the number of false discoveries.
\end{remark}

The minimax mean false discovery result (\ref{minimax1}) has different limits across different sparsity regimes. Specifically, under polynomial sparsity that $s_n = O(n^{\alpha})$ for some $\alpha\in(0,1)$, it converges to 0 as long as the multiplicative constant $c_1> 2/(1-\alpha)$; under linear or near-linear sparsity, it diverges to infinite regardless of the value of $c_1$. 

Other than the number of false positives, the false discovery rate is also of interest in the high dimensional literature. Trivially, we have that
\[
\mbox{FDP}(\widehat \beta):=\frac{\mbox{FP}(\widehat \beta)}{\mbox{FP}(\widehat \beta)+\mbox{TP}(\widehat \beta)}\geq\frac{\mbox{FP}(\widehat \beta)}{\mbox{FP}(\widehat \beta)+s_n},
\]
where $\mbox{TP}(\widehat \beta)$ is the number of true positives.
Intuitively, we have the following approximation
\[
\begin{split}
&  \mbox{FDR}(\widehat\beta):=E_\beta(\mbox{FDP}(\widehat\beta))\geq E\frac{\mbox{FP}(\widehat \beta)}{\mbox{FP}(\widehat \beta)+s_n}\approx \frac{E(\mbox{FP}(\widehat \beta))}{E(\mbox{FP}(\widehat \beta))+s_n}.
\end{split}
\]
Combining with the lower bound result in Theorem \ref{impossibility3}, one can obtain a minimax lower bound for FDR in the following theorem. It is worth noting that by Jensen's inequality, the ``$\approx$'' sign in the above equation is actually ``$\leq$'', hence the lower bound presented in Theorem \ref{minimaxfdr} is not trivial at all, and the rigorous proof is provided in the supplementary material.
\begin{theorem}\label{minimaxfdr}
For the normal means model, the following asymptotic lower bound of FDR holds for the class of rate optimal estimators with $c_1>2$,
\[
\inf_{\widehat\beta\in\Omega_o(c_1)} \sup_{\beta\in\mB(n,s_n)}\mbox{FDR}(\widehat\beta)\geq C (s_n/n)^{[(c_1-2+o^+(1))/2]},
\]
where $C$ is some absolute constant.
\end{theorem}
In Section \ref{fest}, we construct an adaptive estimator $\widehat\beta\in\Omega_o(c_1)$ for any $c_1>2$ that satisfies
$\sup_{\beta\in\mB(n,s_n)}\mbox{FDR}(\widehat\beta)\leq C (s_n/n)^{\tau}$ for any $\tau<c_1/2-1$, therefore it immediately implies

\begin{theorem}[Immediate implication of Theorem \ref{SCmean}]
For any $c_1>2$,
\[
\inf_{\widehat\beta\in\Omega_o(c_1)} \sup_{\beta\in\mB(n,s_n)}\mbox{FDR}(\widehat\beta)\leq C (s_n/n)^{[(c_1-2-o^+(1))/2]}.
\]
\end{theorem}
The above two theorems conclude the following minimax result for FDR under the normal means model: 
\begin{equation}\label{minimaxfdrs}
\inf_{\widehat\beta\in\Omega_o(c_1)} \sup_{\beta\in\mB(n,s_n)}\mbox{FDR}(\widehat\beta)\asymp  (s_n/n)^{[(c_1-2+o(1))/2]}, \mbox{ for any }c_1>2.
\end{equation}

This trivially implies that, when $s_n/n\rightarrow 0$, one is capable to achieve rate optimality and asymptotic zero-FDR simultaneously. But under linear sparsity, the FDR of a rate optimal estimator is bounded away from 0 as $n$ increases. We also want to point out that under linear sparsity, rate minimax estimators may still achieve any arbitrarily small pre-specified positive FDR level (e.g., B-H estimator), but the FDR is not ensured to converge to 0 as $n$ increases.


\begin{remark}
The fundamental result \cite{AbramovichBDJ2006} shows that, there exists a minimax estimator (i.e., B-H estimator) whose FDR convergence is of an order $\log(n)^{-1}$ under polynomial sparsity regime. Our minimax theorem, as a {\it supplementary} result to \cite{AbramovichBDJ2006}, claims that under a broader sparsity regime (polynomial or near-linear sparsity), there always exists a rate minimax estimator $\widehat\beta$ (with suboptimal multiplicative constant) such that its FDR converges as fast as a polynomial rate of $(s_n/n)$.
\end{remark}

We next consider another popular suboptimal convergence rate: $f(n,s_n)=c_1\log n$. Note that $s_nf(n,s_n)\propto s_n\log n$ is a common rate of convergence achieved by many penalized estimators, such as LASSO or SCAD, and is commonly considered to be a near-optimal rate. 
Besides, $f(n,s_n)\geq c\log n$ for some $c>0$ is a necessary condition to ensure that the asymptotic lower bound in Theorem \ref{impossibility3} decays to zero as $n$ increases.
Define
$$\widetilde\Omega_o(c_1)=\{\widehat\beta: \sup_{\beta\in\mB(n,s_n)}E_{\beta}\|\beta-\widehat \beta\|^2\leq c_1 s_n\log n\}$$ as the class of all rate suboptimal estimators. Similarly, by Theorems \ref{impossibility3} and \ref{upper1}, we obtain the following corollary.
\begin{corollary}\label{impossibility3-3}
For the normal means model, the following asymptotic result holds for the class of rate suboptimal estimators,
\begin{equation}\label{minimax-fd2}
\begin{split}
&\inf_{\widehat\beta\in\widetilde\Omega_o(c_1)}\sup_{\beta\in\mB(n,s_n)}E_{\beta} [{\rm FP}(\widehat\beta)] \asymp
C(1-\zeta)n^{1-c_1/2+o(1)}.
\end{split}
\end{equation}
\end{corollary}

Theorem \ref{impossibility3} and Corollary \ref{impossibility3-3} together tell us that, rate suboptimal estimator (with multiplicative constant larger than 2) can achieve $E(\mbox{FP}(\widehat\beta))\rightarrow 0$ regardless of the sparsity growth rate; and vise versa, to guarantee that $E(\mbox{FP}(\widehat\beta))\rightarrow 0$ uniformly over all possible range of sparsity, the estimator is at best sub-optimal. 

\begin{figure}[htp]
 \begin{center}
  \includegraphics[width=13cm]{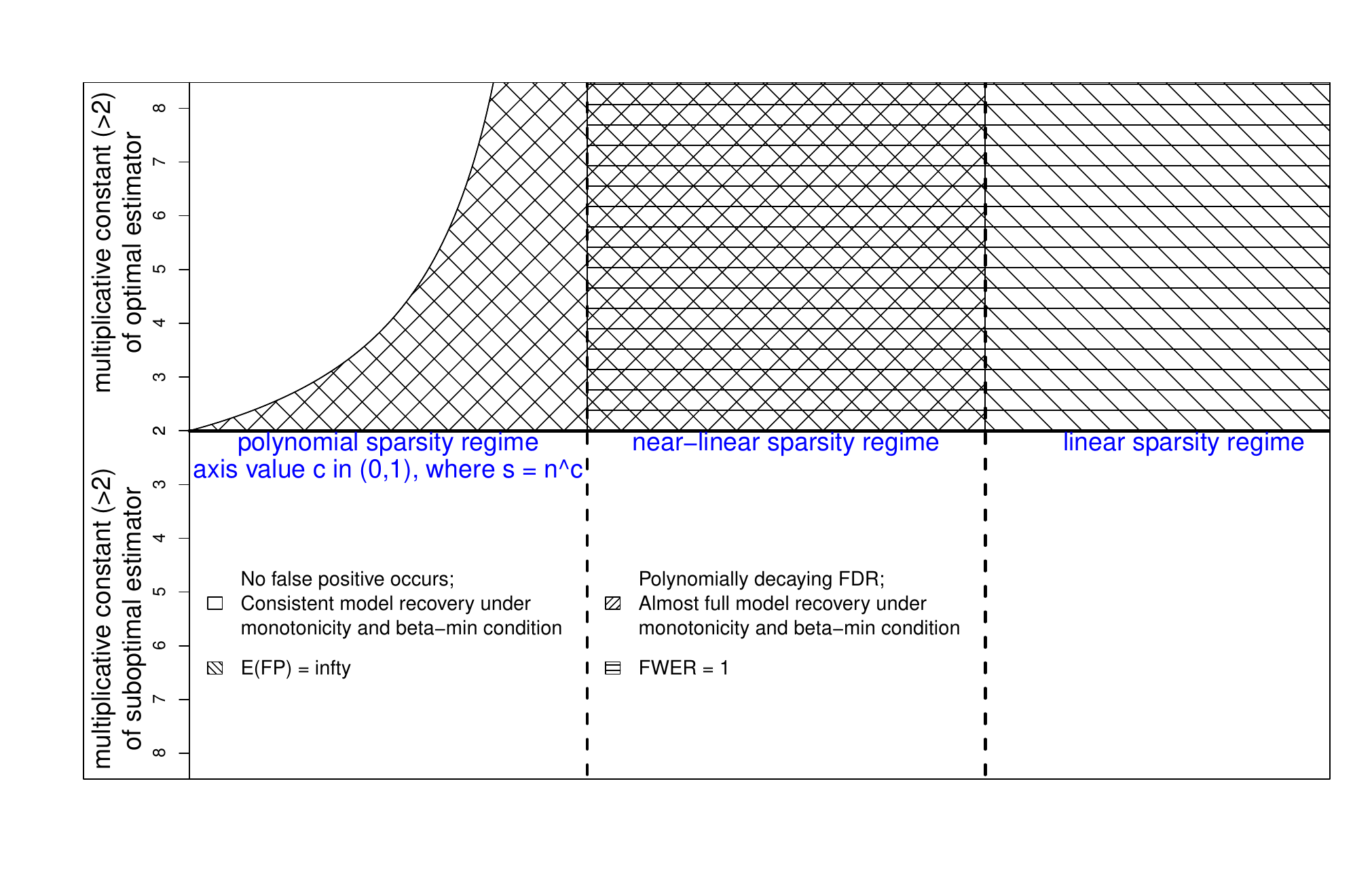}\vskip -0.4in
  \caption{Best possible false discovery control and model selection performance of rate optimal/suboptimal estimators with multiplicative constants greater than 2, under normal means models. }\label{smpplot}
 \end{center}
\end{figure}

Figure \ref{smpplot} aims to visually demonstrate the minimax false discovery control behavior for rate optimal and rate suboptimal estimators under different sparsity regimes. The un-shaded area corresponds to rate optimal estimators with multiplicative constant $c_1>2/(1-\alpha)$ under polynomial sparsity $s_n=O(n^\alpha)$, and rate suboptimal estimators with multiplicative constant $c_1>2$. It represents the regime where one can achieve zero false discovery, as  $\inf_{\widehat\beta}\sup_\beta E(\mbox{FP}(\widehat\beta))=O(n^{-c})$ for some positive $c>0$. Note that this also implies that $\inf_{\widehat\beta}\sup_\beta \mbox{FDR}\rightarrow 0$ and $\inf_{\widehat\beta}\sup_\beta \mbox{FWER}\rightarrow 0$. In opposite, the $\squarebslash$ shaded area corresponds to rate optimal estimators with  multiplicative constant $c_1<2/(1-\alpha)$ under polynomial sparsity or rate optimal estimators under near-linear and linear sparsity, and represents the cases where $E(\mbox{FP}(\widehat\beta))\rightarrow \infty$. Nested within the $\squarebslash$ shaded area, the $\squareslash$ shaded area indicates where rate optimal estimators can achieve polynomially (with respect to sparsity ratio) decaying FDR, despite diverging number of false positives.

Note that zero false positive, i.e., $E(\mbox{FP}(\widehat\beta))\rightarrow 0$, is not equivalent to that family-wise error rate (FWER) $P(\mbox{FP}(\widehat\beta)>0)\rightarrow 0$. So, it is not clear whether suboptimal rate is a necessity for zero FWER control.
Our next Theorem~\ref{impossibility}, which establishes the relationship between FWER rate and convergence speed, confirms it. 

\begin{theorem}\label{impossibility}
For any estimator $\widehat\beta$, we have
\begin{equation}\label{minimax}
\sup_{\beta\in\mB(n,s_n)}P_{\beta}(\mbox{FP}(\widehat\beta)>0)\geq 1-\delta_n- \frac{\sup_{\beta\in\mB(n,s_n)}E_{\beta}\|\widehat\beta-\beta\|^2}{Cs_n\log n}
\end{equation}
for some absolute constant $C$, where $\delta_n\rightarrow 0$ as $n, s_n\rightarrow\infty$.
\end{theorem}

There are two messages delivered by Theorem \ref{impossibility}: 1) to ensure that FWER converges to 0, the estimator is, at best, of the rate $s_n\log n$; and 2) to obtain a convergence rate that is strictly better than $s_n\log n$, one must sacrifice its FWER control in the sense that $\sup_{\beta\in\mB(n,s_n)}P_{\beta}(\mbox{FP}(\widehat\beta)>0)\rightarrow 1$ must hold.

In the end, we summarize the differences between rate optimal and rate suboptimal estimators in terms of false discovery behavior under the nearly-linear and linear sparsity regimes. Within these sparsity regimes, rate suboptimal estimators, at best, can achieve 
$$\sup_{\beta\in\mB(n,s_n)}P_\beta(\mbox{FP}(\widehat\beta)>0)\rightarrow0\;\mbox{and}\;\sup_{\beta\in\mB(n,s_n)}E_\beta(\mbox{FP}(\widehat\beta))\rightarrow0,$$ and vise versa, to achieve such false discovery control, the estimator must be sub-optimal. In contrast, rate optimal estimators $\widehat\beta$ must satisfy that 
\begin{eqnarray}\label{dic}
\sup_{\beta\in\mB(n,s_n)}P_\beta(\mbox{FP}(\widehat\beta)>0)\rightarrow 1\;\mbox{and}\;\sup_{\beta\in\mB(n,s_n)}E_\beta(\mbox{FP}(\widehat\beta))\rightarrow \infty,
\end{eqnarray}
which is reflected by the $\squarehslash$ shaded area in Figure \ref{smpplot}.

The theoretical results in this section (visualized by Figure \ref{smpplot}) carefully depict the gap of false discovery control performance between rate optimal and rate suboptimal estimators. This performance gap increases as true model gets denser. Under polynomial sparsity regime, such a gap can be filled by 
increasing the multiplicative constant of rate minimax estimator, but under near-linear and linear sparsity regimes, this gap becomes inevitable.

Note that result (\ref{dic}) clearly asserts that, under the worse case of $\beta$, rate optimal estimators always select some false discoveries. At first glance, (\ref{dic}) seems to imply that rate optimal estimators is never selection consistent (since consistent selection requires zero false positive selection and zero false negative selection). However, the relationship between rate minimaxity and selection inconsistency is more subtle, and will be elaborated later in Section \ref{impo2}.

\subsection{Linear Regression Models}
In this section, we are interested in generalizing the results derived under normal means models to regression models
$$y=X\beta + \varepsilon,$$ where the design matrix $X\in\BR^{n\times p_n}$ and $\varepsilon\sim N(0, I_n)$.
It is reasonable to assume that the design matrix is asymptotically standardized, in the sense that all column norms in the design matrix converges to $\sqrt n$ uniformly in probability, such that
\begin{itemize}
 \item[C0:] for any $c_3>0$, $P_0(c_3,n):=\min_{1\leq i\leq p_n}P(|\|x_i\|^2/n-1|\leq c_3)\rightarrow 1$, where $x_i$ donates the $i$th column of $X$.
\end{itemize}
Given any function $f(n,p_n,s_n)$, we define the following class of estimators:
$$\Omega_f^l=\left\{\widehat\beta(X,y): \sup_{\beta\in\mB(p_n,s_n)}E_XE_{\beta}\|\beta-\widehat \beta\|^2\leq  s_nf(n,p_n,s_n)/n\right\},$$ where
$E_\beta$ denotes the expectation with respect to $y$ conditional on the random design matrix $X$ under the true parameter $\beta$, and $E_X$ denotes the expectation over the distribution of $X$.

Intuitively, spurious correlation among the columns of design matrix $X$ increases the chance of selecting false positives. However, in the study of the lower bound of $\mbox{FP}(\widehat\beta)$, we shall consider the easiest scenario for false discovery control, i.e., the design matrix are almost orthogonal. In this case, the linear regression problem approximately reduces to normal means models. Hence, we obtain a similar lower bound result for minimax false positive selections as in Theorem \ref{impossibility3}. 
\begin{theorem}\label{impossibility4}
Under the condition (C0), the following non-asymptotic lower bound for the false positive selections holds:
\begin{equation}\label{lowreg}
\begin{split}
&\inf_{\widehat\beta\in\Omega_f^l}\sup_{\beta\in\mB(p_n,s_n)}E_XE_{\beta}(\mbox{FP}(\widehat\beta))\geq (p_n-s_n+1)\times\\
&\sup_{c_2>1,c_3>0}\left\{P_0(c_3,n)\Phi\left(\Phi^{-1}\left({P_0(c_3,n)-1/c_2}\right)-\{(1+c_3)c_2f(n,p_n,s_n)\}^{1/2}\right)\right\}.
\end{split}
\end{equation}
Asymptotically, if $\lim\sup f(n,p_n,s_n)=\kappa<\infty$, then as $n$ and $p_n$ diverge to infinity,
\[\inf_{\widehat\beta\in\Omega_f^l}\sup_{\beta\in\mB(p_n,s_n)}E_XE_{\beta} [{\rm FP}(\widehat\beta)] \geq H(\kappa)\times(1-\zeta)p_n,\]
where $H(\cdot)$ is the same function specified in Theorem \ref{impossibility3};
if $\lim f(n,p_n,s_n)=\infty$, then 
\[\inf_{\widehat\beta\in\Omega_f^l}\sup_{\beta\in\mB(p_n,s_n)}E_XE_{\beta} [{\rm FP}(\widehat\beta)] \geq C(1-\zeta)p_n\exp\{-[1+o^+(1)]f(n,p_n,s_n)/2\},\]
where $C$ is some absolute constant.
\end{theorem}

This non-asymptotic lower bound is derived by assuming the estimator can have access to some information about the unknown quantity $z_i=y-X_{-i}\beta_{-i}$, where $X_{-i}$ denotes the design matrix without the $i$th column and $\beta_{-i}$ denotes the true $\beta$ without the $i$th entry. In practice, this is impossible unless the design matrix possesses strong orthogonality. For instance, when the design matrix is exactly orthogonal, i.e., all the columns in $X$ are norm-$\sqrt n$ and orthogonal, the information about $z_i$ can be accessed via $y^Tx_i(= z_i^Tx_i)$.  Thus, we can construct the following non-adaptive estimator: $$\widehat \beta_i = \frac{y^Tx_i}{n}1(|y^Tx_i|^2/n>(1-o^+(1))f(n,p_n,s_n))$$ for some $o^+(1)$ term. By the same argument in Theorem \ref{upper1}, we have that the following upper bound for minimax rate which matches the lower bound in Theorem \ref{impossibility4}.

\begin{theorem}\label{upper2}
Under the exactly orthogonal design, if $\lim\sup f(n,p_n,s_n)=\kappa<\infty$, then as $n$ and $p_n$ diverge to infinity,
\[\inf_{\widehat\beta\in\Omega_f^l}\sup_{\beta\in\mB(p_n,s_n)}E_XE_{\beta} [{\rm FP}(\widehat\beta)] = O(p_n);\]
if $\lim f(n,p_n,s_n)=\infty$, then as $n$ and $p_n$ diverge to infinity,
\[\inf_{\widehat\beta\in\Omega_f^l}\sup_{\beta\in\mB(p_n,s_n)}E_XE_{\beta} [{\rm FP}(\widehat\beta)] \leq Cp_n\exp\{-[1-o^+(1)]f(n,p_n,s_n)/2\},\]
where $C$ is some absolute constant.
\end{theorem}


The exactly orthogonal design is usually of less interest, since it essentially reduces to normal means models and eliminates the possibility of that $p_n\gg n$. Thus, beyond the exactly orthogonal design, another popular near-orthogonal design (which will be used to match the lower bound in Section \ref{fest}) is the Gaussian design, i.e.,
\begin{itemize}
 \item[C1:] all entries in the design matrix $X$ are i.i.d $N(0,1^2)$.
\end{itemize}
We will show later that, along with certain dimension restriction, the lower bound of minimax rate can also be attained in probability, under condition (C1).

Similar to our study in Section \ref{nmm}, we will consider the following two choices of $f$ (by replacing $n$ with $p_n$):
$$f(n,p_n,s_n)\propto\log(p_n/s_n)\;\;\mbox{and}\;\;f(n,p_n,s_n)\propto\log p_n.$$
For the former,  we define a collection of rate minimax estimators as $$\Omega_o^l(c_1)=\left\{\widehat\beta: \sup_{\beta\in\mB(p_n,s_n)}E_XE_{\beta}\|\beta-\widehat \beta\|^2\leq c_1 s_n\log(p_n/s_n)/n\right\}$$ for linear regression problems.  

The following Corollary directly follows from Theorems \ref{impossibility4} and \ref{upper2}.
\begin{corollary}\label{impossibility4-2}
Under condition (C0), the following asymptotic result holds,
\[
\begin{split}
&\inf_{\widehat\beta\in\Omega_o^l(c_1)}\sup_{\beta\in\mB(p_n,s_n)}E_XE_{\beta} [{\rm FP}(\widehat\beta)] \geq
Cp_n,\quad\mbox{if } \lim(s_n/p_n)=\zeta>0;\\
&\inf_{\widehat\beta\in\Omega_o^l(c_1)}\sup_{\beta\in\mB(p_n,s_n)}E_XE_{\beta} [{\rm FP}(\widehat\beta)] \geq
Cp_n(s_n/p_n)^{c_1/2+o(1)},\quad\mbox{if } \lim (s_n/p_n)=0,
\end{split}
\]
for some constant $C$. These lower bounds can be attained under exactly orthogonal design.
\end{corollary}

Note that the above two rates can be uniformly expressed as $p_n(s_n/p_n)^{c_1/2+o(1)}$ (if $\lim s_n/p_n>0$, then $p_n(s_n/p_n)^{c_1/2+o(1)}\asymp p_n$ as well).
In Section \ref{fest}, we will show that under the Gaussian design condition (C1) and $s_n\log(p_n/s_n)\prec n^{1/2}$, if $c_1>2$, there exists an $L_0$-based penalized adaptive estimator whose number of false discoveries is bounded by $p_n(s_n/p_n)^{\tau}$ for any $\tau<c_1/2$ with dominating probability. Therefore, we claim that $p_n\left({s_n}/{p_n}\right)^{(c_1+o(1))/2}$ is also the minimax rate for number of false discoveries of rate minimax estimators under Gaussian random design. Similarly to our discussions in Section \ref{nmm}, there also exists a clear trade-off between estimation accuracy and false discovery control via the multiplicative constant $c_1$. 

Similar technique of Theorem \ref{minimaxfdr} leads to the following lower bound for the minimax false discovery rate for linear regression models. 
\begin{theorem}\label{minimaxfdr2}
Under condition (C0), we have the following asymptotic lower bound on FDR,
\[
\inf_{\widehat\beta\in\Omega_o^l(c_1)} \sup_{\beta\in\mB(p_n,s_n)}\mbox{FDR}(\widehat\beta)\geq C (s_n/p_n)^{[(c_1-2+o^+(1))/2]},
\]
where $C$ is some absolute constant.
\end{theorem}

Theorem \ref{sharp} in Section \ref{fest} shows that under (C1), $c_1>2$ and $s_n\log(p_n/s_n)\prec n^{1/2}$, an adaptive estimator $\widehat\beta\in\Omega_o^l(c_1)$ achieves that
$\sup_{\beta\in\mB(p_n,s_n)}\mbox{FDR}(\widehat\beta)\leq C (s_n/p_n)^{\tau}$ for any $\tau<c_1/2-1$, therefore:

\begin{theorem}[Immediate implication of Theorem \ref{sharp}]
Under Gaussian design (C1), if $c_1>2$ and $s_n\log(p_n/s_n)\prec n^{1/2}$
\[
\inf_{\widehat\beta\in\Omega_o^l(c_1)} \sup_{\beta\in\mB(n,s_n)}\mbox{FDR}(\widehat\beta)\leq C (s_n/p_n)^{[(c_1-2-o^+(1))/2]},
\]
where $C$ is some absolute constant.
\end{theorem}

Combining the above two theorems, we claim that
\begin{equation}\begin{split}\label{minimaxfdrlinear}
\inf_{\widehat\beta\in\Omega_o^l(c_1)}\sup_{\beta\in\mB(p_n,s_n)} \mbox{FDR}(\widehat\beta)\asymp (s_n/p_n)^{[(c_1-2+o(1))/2]},
\end{split}\end{equation}
under the Gaussian design condition (C1) and $s_n\log(p_n/s_n)\prec n^{1/2}$.

On the other hand, when considering  $f(n,p_n,s_n)\propto\log p_n$ (i.e., sub-optimal rate) and the corresponding collection of all rate suboptimal estimators $$\widetilde\Omega_o^l(c_1)=\left\{\widehat\beta: \sup_{\beta\in\mB(p_n,s_n)}E_XE_{\beta}\|\beta-\widehat \beta\|^2\leq c_1 s_n\log(p_n)/n\right\},$$ we obtain the following corollary, based on Theorems \ref{impossibility4} and \ref{upper2}:

\begin{corollary}\label{impossibility4-3}
Under the condition (C0), the following asymptotic lower bound holds,
\begin{equation}
\begin{split}
&\inf_{\widehat\beta\in\widetilde\Omega_o^l(c_1)}\sup_{\beta\in\mB(p_n,s_n)}E_XE_{\beta} [{\rm FP}(\widehat\beta)] \geq
C(1-\zeta)p_n^{1-c_1/2+o(1)},
\end{split}
\end{equation}
for some $C>0$. This lower bound can be attained under exactly orthogonal design.
\end{corollary}

When $c_1>2$, the above lower bound reduces to zero. Such a zero-false-discovery result is also achievable by rate suboptimal estimators under Gaussian design. For instance, \cite{Wainwright2009} showed that, under Gaussian design (C1), the LASSO estimator yields no false positive with dominating probability, when $s_n\log p_n=o(n)$ and its tuning parameter $\lambda^2 \geq {4\log p_n/n}$. We also find that a similar result holds for the $L_0$ penalized estimator with penalty $\gamma\|\beta\|_0\log p_n$, if $s_n\log p_n = o(\sqrt{n})$ and $\gamma>2$ (refer to Lemma A.5 in the supplementary material). 

For the completeness of this study, a parallel result of Theorem \ref{impossibility} is established below for Gaussian regression models.


\begin{theorem}\label{impossibility2}
Under the condition (C1), any estimator $\widehat\beta$ must satisfy that
\begin{equation}\label{minimax2}
 \sup_{\beta\in\mB(p_n,s_n)}P_{\beta}(\mbox{FP}(\widehat\beta)>0)\geq 1-\frac{\sup_{\beta\in\mB(p_n,s_n)}E_XE_{\beta}\|\widehat\beta-\beta\|^2}{Cs_n\log p_n/n}-\delta_n,
\end{equation}
for some absolute constant $C$, where $\delta_n$ is some sequence that converges to 0 when $n, s_n, p_n\rightarrow\infty$, $\limsup s_n/p_n\leq 1$ and $\log^2p_n\prec n$.
\end{theorem}

In conclusion, false discovery control behavior and rate of convergence are interdependent for regression estimators under the Gaussian random design. Especially when the sparsity is nearly linear or linear, only the rate suboptimal estimators can asymptotically achieve zero false discovery. In the opposite, optimal estimators, under the worst-$\beta$ case, can always yield false positive selections.

\section{Rate Minimaxity and Selection Consistency}\label{impo2}

The results in Section \ref{impo} show that there essentially exists an incompatibility between rate minimaxity and false discovery elimination. In this section, our goal is to understand the more subtle incompatibility relation between rate minimaxity and selection consistency. For simplicity of theoretical analysis, we only focus on normal means models in this section, and show that under certain monotone regularity, selection consistency and rate minimaxity do not coexist. More specifically, within the near-linear sparsity regime, rate minimax estimator at best achieves almost full sparsity structure recovery.

We first want to clarify that the incompatibility between rate minimaxity and zero false discovery does not imply the incompatibility between rate minimaxity and selection consistency. In the statistical literature, the studies of false discovery control and model selection are usually conducted over two different parameter spaces. The former considers all the $s_n$-sparse $\beta\in\mB(n,s_n)$, while the latter often imposes an additional beta-min condition, in the sense that all nonzero $\beta_j$'s are strictly larger than some positive threshold value which may depend on $n$, $p_n$ and $s_n$. Therefore, a rate optimal estimator, which yields false discoveries when the magnitude of true $\beta$ is small, may still be consistent for model selection when the magnitude of true $\beta$ is large enough.

Here is one simple example to illustrate our claim above. Define an estimator that selects the top $s$ $y_i$'s in terms of their absolute values: $\widehat\beta_j=y_j1(|y_j|\geq |y|_{(s)})$. Here, $s$ is the true sparsity value, and $|y|_{(i)}$ is the (descending) order statistic of $|y_i|'s$. This minimax estimator commits false positive selections with high probability when nonzero $\beta_j$ is relatively small, say of an order $o(\sqrt{\log n})$. However,  when all nonzero $|\beta_j|$'s are greater than $\sqrt{c\log n}$ for some $c>2$, it asymptotically selects the true model.
Another example that does not rely on the unknown true sparsity $s$ is
\begin{equation}\label{exp}
 \begin{split}
\widehat\beta(y) = \underset{\beta}{\mathrm{argmin}}\left[\|y-\beta\|^2+ pe(|\{j: |\beta_j|^2<\gamma\log n\}|)\right]
 \end{split}
\end{equation}
with penalty function $pe(k)=\gamma\sum_{i=1}^k \log(n/i)$ for constant $\gamma> 2$, where $|A|$ denotes the number of elements in set $A$.
Note that by the similar arguments used in the proof of Theorem \ref{SCmean}, we can show that estimator (\ref{exp}) has asymptotic convergence rate of $\sqrt{cs_n\log(n/s_n)}$ with $c\approx\gamma$, but is selection consistent when all nozero $\beta_j$ is greater than $\sqrt{\gamma'\log n}$ where the constant $\gamma'>\gamma$.

Despite the above examples, numerous popular estimators, as mentioned in the introduction section, indeed fail to achieve selection consistency and rate optimality simultaneously over all sparsity regimes. This motivates us to explore the rationales behind this phenomenon and further investigate
that under what general conditions, selection consistency and rate minimaxity indeed are incompatible.

The asymptotic behavior under Hamming loss \citep{ButuceaST2017} reveals that the boundary case for  selection consistency is that when all active $|\beta_i| \approx \omega_1 = \sqrt{2\log(n-s_n)}$, while in contrast our theory indicates that the worst case for rate minimax estimation is when all active $|\beta_i| \approx \omega_2 = \sqrt{c_1\log(n/s_n)}$ where $c_1$ is the multiplicative constant of the $L_2$ convergence rate. Note that under near-linear or linear sparsity, $\omega_1\gg\omega_2$.
Thus, the focus of selection consistent estimators is on the parameter subspace where all nonzero $\beta_i$'s satisfy beta-min condition which must be greater than boundary value $\omega_1$ (i.e., large signal case), but rate optimal estimators give more attention to the parameter subspace where all nonzero $\beta_i$'s are around $\omega_2$ (i.e., weak signal case). 
Therefore, a rate minimax and selection consistent estimator must be 1) rate-$\sqrt{s_n\log(n/s_n)}$ under weak signals, but 2) rate-$\sqrt{s_n}$ under strong signals, since consistent selection reduces it to an oracle parametric problem (Refer to  \cite{ndaoud2019interplay} for more discussions about this phase transition phenomenon).
In other words, this estimator must be adaptive to signal strength and behave differently under strong and weak signal cases. Let us take estimator (\ref{exp}), which is rate minimax and selection consistent, as a heuristic example. Under strong signal situation, say all nonzero $|\beta_i|\gg \sqrt{\gamma\log n}$ and they have been all successfully selected in the model, then false discovery occurs only if $|y_i|\geq\sqrt{\gamma\log n}$ for some $i\notin\xi(\beta)$ where $\xi(\beta)$ denotes the true sparse model. On the other hand, under weak signal situation, say $|y_i|\approx|\beta_i|\approx \sqrt{\gamma\log(n/s_n)}$ for all nonzero $\beta_i$'s and they are already included in the model, then  roughly speaking, the occurrence of false discovery only requires $|y_i|\geq\sqrt{\gamma\log [n/(s_n+1)]}$ for some $i\notin\xi(\beta)$. Therefore, this estimator indeed has different false discovery control mechanisms under different signal strength.

Conversely, an estimator shall not be able to achieve selection consistency and rate minimaxity at the same time if it ignores the differences between strong and weak signal cases. One concrete example is that when the estimator $\widehat\beta$ always obeys ``monotonicity'' regardless of the signal strength. Given two data vectors $y$, $z\in\BR^p$, we call that $y$ majorizes $z$ if $\mbox{sign}(y_i)\mbox{sign}(z_i)\geq 0$ and $|y_i|\geq|z_i|$ for all $i=1,\dots,p$. We call that an estimator $\widehat\beta(y)$ is {\it monotone} if $\widehat\beta(y)$ majorizes $\widehat\beta(z)$ providing that $y$ majorizes $z$. Intuitively, the monotonicity means that a larger $y$, in terms of absolute values, leads to a larger $\widehat\beta(y)$. The monotonicity trivially implies that $\xi(\widehat\beta(y))\supseteq\xi(\widehat\beta(z))$ if $y$ majorizes $z$, where $\xi(\widehat\beta)$ denotes the selected model induced by estimator $\widehat\beta$.
In other words, a larger $y$ yields a larger selected model. From such a point of view, monotonicity seems a natural and appealing property for estimators of normal means models. 
The above two examples, estimators $\widehat\beta_j=y_j1(|y_j|\geq |y|_{(s)})$ and (\ref{exp}), are both non-monotone. To see this, let's define two data vectors $y=(a_1,a_2,0,\dots,0)^T$ and $z=(a_1,a_3,0,\dots,0)^T$ where $a_i$'s satisfy $a_3^2>\gamma \log n> a_1^2>a_2^2 > \gamma(\log n+\log(n/2))/2$. Although data vector $z$ majorizes $y$, it is easy to see that $\xi(\widehat\beta(y))=\{1,2\}$ and $\xi(\widehat\beta(z))=\{2\}$ for the estimator (\ref{exp}), i.e., larger data values actually yield a smaller selected model. 

We now discuss why the property of monotonicity leads to the incompatibility between selection consistency and rate minimaxity.
A simple intuition is that a larger $\beta$ generates larger data (in terms of absolute value), and by monotonicity, it usually leads to more false discoveries. To be more specific, given any two parameter values $\beta^{(1)}$ and $\beta^{(2)}$ of the same sparsity such that $\beta^{(2)}$ majorizes $\beta^{(1)}$, then with reasonable probability, the data $y^{(2)}$ generated under $\beta^{(2)}$ majorizes the data $y^{(1)}$ generated under $\beta^{(1)}$. A  monotone estimator hence satisfies that $\xi(\widehat\beta(y^{(2)}))\supset\xi(\widehat\beta(y^{(1)}))$, i.e., $\widehat\beta(y^{(2)})$ selects no less false discoveries than $\widehat\beta(y^{(1)})$.
On the other hand, if $\widehat\beta$ is rate minimax, it must commit false positive selections for some sparse $\beta^{(1)}$ (due to our results in the previous section), and as well as for any $\beta^{(2)}$ that majorizes $\beta^{(1)}$ (due to the above arguments of monotonicity). If we let $\beta^{(2)}$'s magnitude to be sufficiently large and satisfy the beta-min condition, the selection inconsistency occurs.

Formally, given any $\zeta\in(0,1]$ and a positive function $t(n,s_n)$, we define a class of selection consistent estimators $$\Omega_1=\{\widehat\beta(\cdot): \lim_{n}P_{\beta^{(n)}}[\{\xi(\widehat\beta(y))=\xi(\beta^{(n)})]=1\},$$ 
for any sequence of $\beta^{(n)}\in \mB(n, s_n)$ satisfying $\limsup s_n/n\leq \zeta$ and $\|\beta^{(n)}\|_{\min}\geq t(n,s_n)$, where $t(n,s_n)$ represents the required beta-min condition. Similarly, define another class of estimators which ensure no false discovery
$$\Omega_2=\{\widehat\beta(\cdot): \lim_n P_{\beta^{(n)}}[\mbox{FP}(\widehat\beta(y))=0]=1\},$$ for any sequence of $\beta^{(n)}\in \mB(n,s_n) \mbox{ with } \limsup s_n/n\leq \zeta$. Now, we are ready to state the following lemma. 

\begin{lemma}\label{lem1}
If a monotone estimator $\widehat\beta(\cdot)\in\Omega_1$, then 
$\widehat\beta(\cdot)\in\Omega_2$.
\end{lemma}

Therefore, if a monotone estimator is selection consistent under some beta-min condition, then it asymptotically never yields false discoveries regardless of the magnitude of $\beta$. Due to the result of Theorem 2.3 (i.e., rate minimax estimators always yield false discovery under near-linear or linear sparsity), we claim that it must not be rate minimax. In other words, Lemma 3.1 rigorously establish the incompatibility between rate minimaxity and selection consistency for monotone estimators, {especially under the dense signal regime}. We note that almost all of the popular estimators used in the literature are monotone. For instance, if a separable penalty function, i.e., $pe(\beta)=\sum_{i=1}^n p(\beta_i)$ for some function $p$, is used in the normal means models, then the monotonicity of estimators is equivalent to that the thresholding function $S(y_i):=\arg\min_{\beta_i}(y_i-\beta_i)^2+p(\beta_i)$ is monotone. It is easy to figure out that the thresholding function will be monotone, as long as that $p(\cdot)$ is symmetric and
non-negative, $p(0)=0$, and $p$ is monotone on $(0,\infty)$. Therefore, almost all penalty functions proposed in the literature, including LASSO, non-concave penalties \cite{Zhang2010, FanL2001}, $L_0$ penalty and reciprocal penalty \cite{SongL2014_2}, lead to monotone estimators. Note that FDR estimator \cite{AbramovichBDJ2006} and SLOPE estimator \cite{SuC2016} also belong to the class of monotone estimators.




We now understand that a monotone rate optimal estimator, under near-linear or linear sparsity, cannot ensure model selection consistency even when beta-min condition holds. A natural follow-up question is: what is the best model selection result that a monotone rate optimal estimator $\widehat\beta$ can achieve. For type I error control, the results in Section \ref{nmm} prove that, if $s_n\prec n$, then $\mbox{minimax }E(\mbox{FP})\asymp n(s_n/n)^c  \asymp s_n(s_n/n)^{c-1}$ for some $c\approx c_1/2$; if $s_n\asymp n$, then $\mbox{minimax }E(\mbox{FP})\asymp n \asymp s_n$. As for the type II error control, we conduct a rather trivial analysis as follows. When the true parameter $\beta$ satisfies a strong beta-min condition\footnote{Note that a sharper beta-min condition, that guarantees that rate optimal estimators yield zero false negative, can be further investigated but is beyond the scope of this work.} such that $\|\beta\|_{\min}\succ \sqrt{s_n\log(n/s_n)}$, it is easy to verify that $\mbox{FN}(\widehat\beta)\rightarrow 0$,  where the convergence holds in $L_1$ and in probability, since $E\|\widehat\beta-\beta\|^2=O(s_n\log(n/s_n))$. Combining the above discussions regarding type I and type II error control, the best a monotone rate minimax estimator can possibly do, in terms of model selection under the above beta-min condition, is that: 1) under polynomial sparsity, both $\mbox{FP}(\widehat\beta)$ and $\mbox{FN}(\widehat\beta)$ converge to 0 in probability, and thus $\widehat\beta$ recovers the exact sparsity structure; 2) under near-linear sparsity, $[\mbox{FP}(\widehat\beta)+\mbox{FN}(\widehat\beta)]/s_n=|\xi(\widehat\beta)\bigtriangleup\xi(\beta)|/s_n\rightarrow 0$ in probability, where $\bigtriangleup$ denotes the symmetric difference of two sets. That is, $\widehat\beta$ almost fully recovers the sparsity structure \cite{ButuceaST2017}; 3) Under linear sparsity, it can only ensure that 
$\xi(\widehat\beta)\supsetneq\xi(\beta)$ and $|\xi(\widehat\beta)|/|\xi(\beta)|\leq (1+\delta)$ for some $\delta>0$.
Cases 1) and 2) correspond to the unshaded area and $\squareslash$ shaded area in Figure \ref{smpplot}, respectively.




\begin{remark}
 The concept of monotonicity directly applies to the linear regression models where the columns of the design matrix are $\sqrt n$-norm and orthogonal. Under orthogonal design, $y=X\beta+\varepsilon$ can be re-written as $X^Ty=\beta+X^T\varepsilon$. Thus, we call that an estimator $\widehat\beta(X,y)$ is monotone if $\widehat\beta$ is monotone with respect to $X^Ty$ (i.e., for any two $y^{(1)}$ and $y^{(2)}\in\BR^n$, $\widehat\beta(X,y^{(1)})$ majorizes $\widehat\beta(X,y^{(2)})$ when $X^Ty^{(1)}$ majorizes $X^Ty^{(2)}$). Hence, our previous discussions and arguments also hold for linear regression models under the orthogonal design. However, under more general regression settings, due to the possible column dependencies within $X$, this monotonicity concept cannot be easily generalized. More discussions on the incompatibility between rate minimaxity and selection consistency under regression model can be found in Section B of the supplementary material.
\end{remark}

\section{Adaptive Optimal Penalized Estimator}\label{fest}
In this section, we  will construct {\em adaptive} rate minimax estimators that are capable of achieving the lower bounds of false positive selection and false discovery rate derived in Section \ref{impo}. Therefore, the polynomial decay of the false discovery control is attainable. Our construction applies to both normal means models and linear regression models under Gaussian design.
It is worth emphasizing that these estimators merely serve as the ``proof of concept'', justifying that these lower bounds are achievable by adaptive estimators. This work doesn't promote the use of these estimators in practice because it involves NP-hard optimizations.

We consider a class of estimators based on $L_0$ selection criterion: $$\widehat\beta =\underset{\beta}{\mathrm{argmin}}\|y-X\beta\|^2+pe(\|\beta\|_0),$$ where the penalty $pe(\cdot)$ only depends on the $L_0$ norm of $\beta$, and $X=I$ for normal means models.
Equivalently, we first search a subset model which minimizes the following criterion:
$$\widehat\xi=\underset{\xi\subset\{1,\dots,p_n\}}{\mathrm{argmin}}\mbox{RSS}(\xi)+pe(|\xi|),$$ where $\mbox{RSS}(\xi)$ is the residual sum of squares under model $\xi$, i.e., $$\mbox{RSS}(\xi):=y^T(I-X_\xi(X_{\xi}^TX_{\xi})^{-1}X_{\xi}^T)y.$$ Then the estimator $\widehat \beta$ will be the OLS estimation based on this selected model $\widehat\xi$.

The particular penalty function we consider in this section is
$$pe(k) = \gamma\sum_{i=1}^k\log(p_n/i)\;\;\mbox{if}\;k\leq \widetilde p_n\;\;\mbox{and}\;\;pe(k)=\infty\;\;\mbox{if}\;k>\widetilde p_n$$  for some user-specific parameter $\widetilde p_n$ and $\gamma$. Therefore, the estimator reduces to
\begin{equation}\label{SCest}
\widehat\beta=\underset{\|\beta\|_0< \widetilde p_n}{\mathrm{argmin}} \|y-X\beta\|^2+\gamma\sum_{i=1}^{\|\beta\|_0}\log(p_n/i).
\end{equation}
A trivial choice for $\widetilde p_n$ is $\min(p_n, n)$. Unlike the classic BIC and AIC criteria,  this penalty function is not proportional to the number of selected variables. Instead, it assigns smaller penalty for adding one more covariate into the current model when the current model size is larger. This form of penalty has been extensively used in the literature  \cite{GeorgeFoster2000,TibshiraniK1999, WuZ2013,FosterS1999,BirgeM2001}. For example,
\cite{WuZ2013} established sharp minimaxity of estimation when $\gamma=2$ and $\widetilde p_n = n/\log n$ under the normal means models.
More generally, \cite{BirgeM2001} studied a large class of selection penalization including (\ref{SCest}). These existing results mostly focused on the convergence rate of $\|\widehat\beta-\beta^{(n)}\|$ or the risk $\|X\widehat\beta-X\beta^{(n)}\|$. 
Other than these, in this section, we will also focus on its selection behavior, especially its false discovery control behavior.

\begin{remark}
The penalty used in (\ref{SCest}) is closely related to the B-H estimator \cite{BenjaminiH1995}. Under the normal means model, as pointed out by \cite{AbramovichBDJ2006}, B-H estimator, which targets level $q$ FDR, is equivalent to the some local optimum of penalized likelihood with penalty $\sum_{l=1}^{\|\beta\|_0}[\Phi^{-1}(1-q l/2p_n)]^2$. This penalty approximately equals to $\gamma\sum_{i=1}^{\|\beta\|_0}\log(p_n/i)$ with $\gamma=2$, since $[\Phi^{-1}(1-q l/2p_n)]^2\sim 2\log(p_n/l)-2\log q-\log\log(p_n/ql)$. Another related work is the SLOPE estimator \cite{SuC2016,BogdanvSSC2015}, which can be viewed as a soft thresholding FDR penalization.
\end{remark}


For the sparse means model, the next theorem shows that the $\sum_{i=1}^{k}\log(p/i)$ penalty induces a rate minimax estimator,
as well as polynomial decay of false discovery control.
\begin{theorem}\label{SCmean}
Consider the estimator (\ref{SCest}) under the normal means model (i.e., $X=I$) with $\beta^{(n)}\in\mB(n,s_n)$ and $\limsup s_n/n\leq \zeta\in[0,1)$. If the tuning parameters $ \widetilde p_n\in[s_n,n]$ and $\gamma$ is sufficiently large, then
  the following properties  hold with dominating probability,
 \begin{equation*}
 \|\widehat\beta-\beta^{(n)}\|^2\leq (\gamma+\delta) s_n\log(n/s_n) \mbox{ and }
\end{equation*}
\begin{equation*}
 \frac{\mbox{FP}(\widehat\beta)}{\mbox{TP}(\widehat\beta)}\leq \left(\frac{\mbox{TP}(\widehat\beta)}{n}\right)^{\gamma'}\leq (s_n/n)^{\gamma'} 
 \mbox{ if } \mbox{TP}(\widehat\beta)>0; \quad \mbox{FP}(\widehat\beta)=0 \mbox{ if } \mbox{TP}(\widehat\beta)=0,
\end{equation*}
 where $\delta$ and $\gamma'$ are some positive constants.
Furthermore, if $s_n\prec n$ and tuning parameters satisfy $s_n\leq\widetilde p_n\prec n$ and $\gamma>2$, then the above results asymptotically
hold for any $\delta>0$ and $0<\gamma'\leq (\gamma-2)/2$.
\end{theorem}
This result implies that in probability, the $L_2$ convergence of $\widehat\beta$ is bounded by minimax rate with a multiplicative constant $\gamma+\delta$, and that the number of false discoveries of $\widehat\beta$ is bounded by $s_n(s_n/n)^{\gamma'}=n(s_n/n)^{\gamma'+1}$, which matches the lower bound for the minimax number of false discoveries presented in Section \ref{nmm}, except for the polynomial degree. Furthermore, when  $s_n\prec n$  and the hyperparameter $\widetilde p_n$ is chosen such that $s_n\leq\widetilde p_n\prec n$, the polynomial degree can be as large as 
$\gamma'+1=\gamma/2= c_1/2-\delta/2$, for any arbitrarily small constant $\delta$. Compared with the polynomial degree of the minimax result in Corollary \ref{impossibility3-2}, we claim that this estimator (\ref{SCest}) achieves the minimax rate for the number of false discoveries, in probability.

Theorem~\ref{SCmean} also implies that the \mbox{FDP} is bounded by $(s_n/n)^{\gamma'}/[(s_n/n)^{\gamma'}+1]\leq(s_n/n)^{\gamma'}$ in probability. As for the false discovery rate, we note that, the phrase ``with  dominating probability'' in the theorem statement means ``with probability as least $1-\exp\{Cs_n\log(n/s_n)\}$ for some $C>0$'' (as shown in the proof in the supplementary material). Hence, we have 
\[
\mbox{FDR}=E(\mbox{FDP})\leq (s_n/n)^{\gamma'}+\exp\{-Cs_n\log(n/s_n)\}\asymp
\begin{cases}
 (s_n/n)^{\gamma'} & \mbox{if }1\prec s_n\prec n\\
 \mbox{constant} & \text{if } s_n\asymp n
\end{cases}
.
\]
Since $\gamma'$ can be sufficient close to (but larger than) $c_2/2-1$ when $s_n\leq \tilde p_n\prec n$,
the false discovery rate of the estimator (\ref{SCest}) is rate minimax.
In conclusion, the rate optimal estimator (\ref{SCest}) achieves the best possible type I error control rate, in both the number of false discoveries and the false discovery rate.

For this estimator, the value of $\gamma$ plays a role of balancing the rate of convergence and the rate of false discovery control when $\widetilde p_n\prec n$. A larger $\gamma$ leads to a larger polynomial degree $\gamma'$, but at the expense 
of a greater multiplicative constant in the convergence rate. In other words, the trade-off between false discovery control and estimation accuracy is tuned through the choice of $\gamma$.


As stated in the next theorem, similar results hold for Gaussian design regression models.
\begin{theorem}\label{sharp}
Consider the estimator (\ref{SCest}) for Gaussian design regression models with $\beta^{(n)}\in\mB_0(p_n,s_n)$, $s_n\log(p_n/s_n)\prec n$ and $\limsup s_n/p_n\leq \zeta\in[0,1)$. If we choose the tuning parameter $\widetilde p_n\in[s_n \min(n,p_n)]$ and $\gamma$ to be a sufficiently large constant, then the following results hold with dominating probability:
 \begin{equation}\label{r1}
 n\|\widehat\beta-\beta^{(n)}\|^2\leq (\gamma+\delta)s_n\log(p_n/s_n),
\mbox{ and }
 \mbox{FP}(\widehat\beta)\leq \delta's_n,
\end{equation}
for some positive constants $\delta$ and $\delta'$. If furthermore, $s_n^2\log^2(p_n/s_n)\prec n$, then
\begin{equation}\label{r2}
 \frac{\mbox{FP}(\widehat\beta)}{\mbox{TP}(\widehat\beta)}\leq \left(\frac{\mbox{TP}(\widehat\beta)}{p_n}\right)^{\gamma'}\leq (s_n/p_n)^{\gamma'}
 \mbox{ if } \mbox{TP}(\widehat\beta)>0; \quad \mbox{FP}(\widehat\beta)=0 \mbox{ if } \mbox{TP}(\widehat\beta)=0,
\end{equation}
 hold with dominating probability for some constant $\gamma'$.
 
 More importantly, under the strictly sparse setting, i.e., $s_n\prec p_n$, if we choose 
 $s_n\leq\widetilde p_n\prec p_n$ and $\gamma>2$, then all previous results hold  asymptotically in probability
 for any positive $\delta$, $\delta'>0$ and $0<\gamma'\leq (\gamma-2)/2$. Note that it also implies that 
 estimator $\widehat\beta$ is almost sharply minimax when $\gamma$ is sufficiently close to 2 (refer to the discussion in below).
\end{theorem}

When $s_n\log(p_n/s_n)\ll n$, the estimator (\ref{SCest}) is rate minimax for Gaussian linear regression, and its number of false discoveries is bounded linearly by $\delta' s_n$. Under strict sparsity ($s_n\prec p_n$), the multiplicative constant of its convergence rate can be arbitrarily close to 2, and $\mbox{FP}/s_n$ can be arbitrarily close to 0, when $\gamma\approx 2$. As shown by Theorem 1.3 of \cite{SuC2016}, the minimax $L_2$ convergence rate for Gaussian linear regression is $\{[2+o(1)]s_n\log(p_n/s_n)/n\}^{1/2}$, thus this estimator is almost sharply minimax.

Under additional dimensionality condition $s_n^2\log^2(p_n/s_n)\prec n$, 
we can achieve polynomial control for type I errors: in probability, the number of false discoveries is bounded by $p_n(s_n/p_n)^{\gamma'+1}$, and the FDP is bounded by $(s_n/p_n)^{\gamma'}$. Given the strict sparsity setting and $s_n\leq\widetilde p_n\prec p_n$, the polynomial degree $\gamma'$ can be arbitrarily close to $c_1/2-1$, where $c_1\approx \gamma$ is the multiplicative constant of convergence rate of this estimator. Similar remarks on the rate of FDR for the Theorem \ref{SCmean} apply here as well, and we claim that the number of false discoveries and false discovery rate of estimator (\ref{SCest}) attain minimax rate under Gaussian design.

\begin{remark}
The dimensionality assumption $s_n\log(p_n/s_n)\prec \sqrt{n}$ provides stronger regularity on the orthogonality of design matrix than the commonly used assumption $s_n\log p_n\prec n$. Under this assumption, random matrix theory \cite{Vershynin2012} ensures that
the singular values of any sub-matrix of $X/\sqrt{n}$ consisting of $O(s)$ columns are within $[1-o((s\log(p/s))^{-1/2}),1+o((s\log(p/s))^{-1/2})]$, such that the existence of false negative predictors (whose true regression coefficients are up to $O(\sqrt{s\log(p/s)/n})$) doesn't induce additional false positives via the spurious correlation.
\end{remark}

\section{Conclusion and Discussion}\label{diss}

This work investigates how the rate of convergence and the selection behavior affect each other under the high dimensional setting.
There is a subtle trade-off between estimation accuracy and type I error control. This trade-off comes into effect particularly under near-linear or linear sparsity, where rate optimal estimators can yield an unbounded number of false discoveries and zero-false-discovery estimators must be suboptimal. 

In the sense of minimax (i.e., under the worse scenario), we show that the rate optimal estimators induce as many as $p_n(s_n/p_n)^{\tau}$ false positive selections, and its FDR decays as a polynomial of the sparsity ratio $s_n/p_n$. Such a polynomial rate of false discovery control can be achieved by the adaptive penalty $\gamma\sum_{i}^{k}\log(p_n/i)$.
Note that the SLOPE estimator employs a soft $L_1$ version of $\gamma\sum_{i}^{k}\log(p_n/i)$-penalization. Hence, we conjecture that the SLOPE estimator can also achieve a similar asymptotic FDP control as in our Theorem \ref{sharp}, while enjoying a much better computational efficiency than the estimator \eqref{SCest}. Readers may refer to \cite{KosB2019} for some recent theoretical investigation of the SLOPE estimator.

In this work, we only consider sparse estimators that contain exact 0 coordinates. On the other hand, shrinkage estimators, which don't contain exact zero entries, have been commonly used for sparse linear models as well, e.g., \cite{BhattacharyaPPD2015,SongL2017,VanKV2014}. A shrinkage estimator $\widehat\beta$ can also induce model selection result, via a simple truncation as $\tilde\beta_i=\widehat\beta_i1(|\widehat\beta_i|\leq \eta_{n,i})$ for some (data-dependent) $\eta_{n,i}$.  
 If these thresholding values $\eta_{n,i}$'s are sufficiently small, as in \cite{BhattacharyaPPD2015} and \cite{SongL2017}, then $\tilde\beta$ and $\widehat\beta$ have negligible difference between their convergence rates.
 Therefore, our result on the relationship between false discovery control and convergence rate of an exactly sparse estimator, can naturally extends to the relationship between false discovery control of $\tilde\beta$ and convergence rate of $\widehat\beta$ as well.


\bibliographystyle{imsart-number}
\bibliography{ref}

\end{document}


\thispagestyle{empty}
\title{Supplementary Material of ``Optimal False Discovery Control of Minimax Estimator''}
\author{Qifan Song and Guang Cheng}

\date{}

\maketitle
\appendix
\section{Proofs}
{\noindent \bf Proof of Lemma 3.1}
\begin{proof}
If a monotone selection consistent estimator $\widehat\beta(\cdot)\notin\Omega_2$, then there exists a sequence of $\beta^{(n)}\in \BR^n$, such that $\lim\sup_nP_{\beta^{(n)}}(K_n)=c>0$,
where $K_n=\{y: \xi(\widehat\beta(y))\backslash\xi(\beta^{(n)})\neq\emptyset\}\subset\BR^n$. To induce contraction, 
it is sufficient to show that there exists a sequence of $\widetilde\beta^{(n)}$, such that $\|\widetilde\beta^{(n)}\|_{\min}\geq t(n,s_n)$, $\xi(\widetilde\beta^{(n)})=\xi(\beta^{(n)})$
and $P_{\widetilde\beta^{(n)}}(K_n)\geq P_{\beta^{(n)}}(K_n)$.

For any $\beta^{(n)}$, if $0<|\beta^{(n)}_i|<t(n,s_n)$ for some index $i$, we claim that we can replace this entry by some sufficiently large value, i.e., there exists a $\check\beta^{(n)}$
such that $\check\beta^{(n)}_j=\beta^{(n)}_j$ for all $j\neq i$ and $|\check\beta^{(n)}_i|>t(n,s_n)$, such that $P_{\check\beta^{(n)}}(K_n^c)\leq P_{\beta^{(n)}}(K_n^c)$.
Therefore, the $\widetilde\beta^{(n)}$ can be constructed via sequentially replacing all small-but-nonzero entries of $\beta^{(n)}$ with large absolute value entries.

Now let us prove the existence of $\check\beta^{(n)}$. Without losing generality, let the index $i=1$ and denote $\beta^{(n)}=(\beta_1^{(n)},\beta_{-1}^{(n)})'$ where $\beta_{-1}^{(n)}\in\BR^{n-1}$.
We define set $K_n^c(\lambda)=\{y\in \BR^{n-1}: (\lambda,y)^T\in K_n^c\}$ where $\lambda\in\BR$.
By the monotonicity of estimator, if $sign(\lambda_1)=sign(\lambda_2)$, and $|\lambda_1|\leq |\lambda_2|$, then $K_n^c(\lambda_1)\supseteq K_n^c(\lambda_2)$.
First of all, if $P_{\beta^{(n)}}(K_n^c)=0$, then $K_n^c$ is a zero Lebesgue measure set, and the existence of $\check\beta^{(n)}$ is trivial.
Now we only consider  the case that $P_{\beta^{(n)}}(K_n^c)>0$. If 
 \begin{equation}\label{pos}
  \lim_{\lambda\rightarrow\infty}\int_{\lambda}^{+\infty}\int_{K_n^c(z)}\phi(z';{\beta^{(n)}_{-1}}, I_{n-1})dz' \phi(z;{\lambda},1) dz=0, 
 \end{equation}
where $\phi(\cdot;\mu,\Sigma)$ denotes the density of a (multivariate) normal distribution, then 
for any $\epsilon>0$, there exists a $\lambda_0>0$ such that 
$\int_{\lambda_0}^{+\infty}\int_{K_n^c(z)}\phi(z';{\beta^{(n)}_{-1}}, I_{n-1})dz' \phi(z;{\lambda_0},1) dz\leq \epsilon$. Now we pick 
a sufficiently large $\lambda_1>\max(\lambda_0,t(n,s_n))$ such that 
$\int_{-\infty}^{\lambda_0}\int_{K_n^c(z)}\phi(z';{\beta^{(n)}_{-1}}, I_{n-1})dz' \phi(z;{\lambda_1},1) dz
< \int_{-\infty}^{\lambda_0} \phi(z;{\lambda_1},1) dz
< \epsilon:=P_{\beta^{(n)}}(K_n^c)/3$. Furthermore, 
\begin{equation}\label{inf}
\begin{split}
 &\int_{\lambda_0}^{+\infty}\int_{K_n^c(z)}\phi(z';{\beta^{(n)}_{-1}}, I_{n-1})dz' \phi(z;{\lambda_1},1) dz\\
 =& \int_{\lambda_0}^{\lambda_1}\int_{K_n^c(z)}\phi(z';{\beta^{(n)}_{-1}}, I_{n-1})dz' \phi(z;{\lambda_1},1) dz+\int_{\lambda_1}^{+\infty}\int_{K_n^c(z)}\phi(z';{\beta^{(n)}_{-1}}, I_{n-1})dz' \phi(z;{\lambda_1},1) dz\\
 \leq& \int_{\lambda_0}^{\lambda_1}\int_{K_n^c(z)}\phi(z';{\beta^{(n)}_{-1}}, I_{n-1})dz' \phi(z;{\lambda_0},1) dz+   \int_{\lambda_0}^{+\infty}\int_{K_n^c(z)}\phi(z';{\beta^{(n)}_{-1}}, I_{n-1})dz' \phi(z;{\lambda_0},1) dz\\
 < &2\epsilon,
 \end{split}
\end{equation}
where the first inequality is due to the fact that $K_n^c(z)$ is a smaller set as $z$ increase. 
Hence, we can choose $\check\beta^{(n)}=(\lambda_1,\beta^{(n)}_{-1})$, and it satisfies $P_{\check\beta^{(n)}}(K_n^c) < \epsilon+2\epsilon =  P_{\beta^{(n)}}(K_n^c)$.

Similarly, if 
 \begin{equation}\label{neg}
  \lim_{\lambda\rightarrow-\infty}\int_{-\infty}^{\lambda}\int_{K_n^c(z)}\phi(z';{\beta^{(n)}_{-1}}, I_{n-1})dz' \phi(z;{\lambda},1) dz=0, 
 \end{equation}
we can construct the $\check\beta^{(n)}=(\lambda_1,\beta^{(n)}_{-1})$, where $\lambda_1<0$, and $|\lambda_1|$ is sufficiently large.

If both (\ref{pos}) and (\ref{neg}) fail, then we must have $K_{n+}^c=\cap_{z>0}K_n^c(z)\neq\emptyset$ and $K_{n-}^c=\cap_{z<0}K_n^c(z)\neq\emptyset$ (Note that the two limits $K_{n+}^c$ and $K_{n-}^c$ always exist since $K_n^c(z)$ becomes smaller as $|z|$ increases).
 Without losing generality, we assume that $\int_{K_{n-}^c}\phi(z';{\beta^{(n)}_{-1}}, I_{n-1})dz'\geq\int_{K_{n+}^c}\phi(z';{\beta^{(n)}_{-1}}, I_{n-1})dz':=c_K$, and it also implies that $P_{\beta^{(n)}}(K_n^c)\geq c_K$. Then,
 \begin{itemize}
     \item If $ P_{\beta^{(n)}}(K_n^c)=\int_{-\infty}^{+\infty}\int_{K_n^c(z)}\phi(z';{\beta^{(n)}_{-1}}, I_{n-1})dz' \phi(z;\beta^{(n)}_1,1) dz=c_K$, since $K_n^c(z)$ is shrinking
as $|z|$ increases, we must have that $K_n^c\backslash(\BR^+\otimes K_{n-}^c \cup \BR^-\otimes K_{n+}^c)$ is zero-measure, and $\int_{K_{n-}^c}\phi(z';{\beta^{(n)}_{-1}}, I_{n-1})dz'=\int_{K_{n+}^c}\phi(z';{\beta^{(n)}_{-1}}, I_{n-1})dz'=c_K$.
Therefore, for any $\lambda_1$ and $\check\beta^{(n)}=(\lambda_1,\beta^{(n)}_{-1})$,  we have $ P_{\check\beta^{(n)}}(K_n^c)= P_{\check\beta^{(n)}}(BR^+\otimes K_{n-}^c \cup \BR^-\otimes K_{n+}^c) =  P_{\beta^{(n)}}(K_n^c)$.
\item If $ P_{\beta^{(n)}}(K_n^c)>c_K$, then we must have that
 \begin{equation}
  \lim_{\lambda\rightarrow\infty}\int_{\lambda}^{+\infty}\int_{K_n^c(z)\backslash K_{n+1}^c}\phi(z';{\beta^{(n)}_{-1}}, I_{n-1})dz' \phi(z;{\lambda},1) dz=0.
 \end{equation}
By the same arguments of (\ref{inf}), there exists a sufficiently large $\lambda_1>0$, and $\check\beta^{(n)}=(\lambda_1,\beta^{(n)}_{-1})$ 
such that $P_{\check\beta^{(n)}}(K_n^c\backslash \mathbb{R}\otimes K_{n+1}^c)\leq P_{\beta^{(n)}}(K_n^c\backslash \mathbb{R}\otimes K_{n+1}^c)$. This further imples that $P_{\check\beta^{(n)}}(K_n^c)\leq P_{\beta^{(n)}}(K_n^c)$
 \end{itemize}

This concludes the proof.
\end{proof}

\begin{lemma}\label{nplemma}
Assume $h(x)\in[0,1]$, if $\int h(x)\phi(x;\mu_0,1)\geq a$, then $\int h(x)\phi(x;0,1)\geq \Phi[\Phi^{-1}(a)-\mu_0]$.
\end{lemma}
\begin{proof}
Let's consider two simple hypotheses $H_0: \mu = \mu_0$ and $H_1: \mu = 0$ with one observation $X\sim N(\mu,0)$, and view $h\in[0,1]$ as the testing function, such that the test size is  $\int [1-h(x)]\phi(x;\mu_0,1)\leq 1- a$, and power is  $\int [1-h(x)]\phi(x;0,1)$.

By Neyman-Pearson (N-P) Lemma, the UMP test for all size $1-a$ tests is the likelihood ratio test. (Note that the standard N-P Lemma applies to all testing functions $h\in\{0,1\}$, but its proof can easily be extend to the case of all $h\in[0,1]$.)
This likelihood ratio test is denoted by $h_{LT}=1(\phi(x;\mu_0,1)/\phi(x;0,1)\geq \Lambda)$, where $\Lambda$ satisfies $ \int [1-h_{LT}(x)]\phi(x;\mu_0,1)\leq 1- a$. It is easy to figure that $h_{LT}=1(x\geq \mu_0+\Phi^{-1}(1-a))$, and its corresponding power is $\Phi(\mu_0+\Phi^{-1}(1-a))$.

Since $h_{LT}$ has the largest power, $\int [1-h(x)]\phi(x;0,1)\leq \Phi(\mu_0+\Phi^{-1}(1-a))$, i.e.,
$\int h(x)\phi(x;0,1)\geq 1-\Phi(\mu_0+\Phi^{-1}(1-a))=\Phi(\Phi^{-1}(a)-\mu_0)$.
\end{proof}


{\noindent \bf Proof of Theorem 2.1}
\begin{proof}
We denote $\widehat\eta_i=1\{\widehat\beta_i\neq 0\}$ be the selector induced by estimator $\widehat\beta$,
$\eta_i=1\{\beta_i\neq 0\}$ and the prior $\pi_1$ be a uniform distribution on set $\mB_1=\{\beta\in \BR^n: \beta_i=0\mbox{ or }\sqrt{c_2f(n,s_n)}, \|\beta\|_0=s_n-1\}$ for some constant $c_2\geq 1$.
Let $E_{\beta}$ be the expectation over $y$ with respect to distribution $y\sim N(\beta,I)$, let $E_{\pi_1}$ be the expectation 
over $\beta$ with respect to the prior distribution $\pi_1$. 
We use subscript ``${-i}$'' to denote subset $\{1,\dots,i-1,i+1,\dots,n\}$, and we use $\pi_1(\beta_{-i}|\beta_i)$ and $\pi_1(\beta_{i})$ to denote conditional prior and marginal prior respectively.
\[\begin{split}
&\sup_{\beta\in\mB(n,s_n)}E_{\beta} \sum_{i=1}^n \widehat\eta_i1\{\eta_i=0\}\geq E_{\pi_1}E_{\beta}\sum_{i=1}^n \widehat\eta_i1\{\eta_i=0\}\\
= & \sum_{i=1}^n\int_\beta\int_y\widehat\eta_i(y)1\{\eta_i=0\}\prod_{i=j}^n\phi(y_j;\beta_j,1)dy
\pi_1(\beta)d\beta\\
= & \sum_{i=1}^n\int_{\beta_{i}}\int_{\beta_{-i}}\int_{y_i}\int_{y_{-i}}\widehat\eta_i(y)\phi(y_i;\beta_i,1)\prod_{j\neq i}^n\phi(y_j;\beta_j,1)dy_{-i}dy_i
\pi_1(\beta_{-i}|\beta_i)1\{\eta_i=0\}\pi_1(\beta_i)d\beta_{-i}d\beta_{i}\\
= & \sum_{i=1}^n\int_{\beta_{i}}\int_{\beta_{-i}}\int_{y_i}\int_{y_{-i}}\widehat\eta_i(y)\phi(y_i;\beta_i,1)\prod_{j\neq i}^n\phi(y_j;\beta_j,1)dy_{-i}dy_i
\widetilde\pi_1(\beta_{-i})1\{\eta_i=0\}\pi_1(\beta_i)d\beta_{-i}d\beta_{i}\\
= & \sum_{i=1}^n\int_{\beta_{i}}\int_{y_i}\phi(y_i;\beta_i,1)\int_{\beta_{-i}}\int_{y_{-i}}\widehat\eta_i(y)\prod_{j\neq i}^n\phi(y_j;\beta_j,1)dy_{-i}
\widetilde\pi_1(\beta_{-i})d\beta_{-i}dy_i1\{\eta_i=0\}\pi_1(\beta_i) d\beta_{i}\\
\end{split}
\]
where $\widetilde\pi_1(\beta_{-i})=\pi_1(\beta_{-i}|\beta_i=0)$, such that $\widetilde\pi_1(\beta_{-i})1\{\eta_i=0\}=
\pi_1(\beta_{-i}|\beta_i)1\{\eta_i=0\}$. Note the 
$\widetilde \pi_1$ is a uniform distribution over 
set $\mB_3=\{\beta\in \BR^{n-1}: \beta_i=0\mbox{ or }\sqrt{c_2f(n,s_n)}, \|\beta\|_0=s_n-1\}$. 

Therefore, we let $\widetilde\eta_i(y_i) =\int_{\beta_{-i}}\int_{y_{-i}}\widehat\eta_i(y)\prod_{j\neq i}^n\phi(y_j;\beta_j,\sigma^2)dy_{-i}
\widetilde\pi_1(\beta_{-i})d\beta_{-i} $, which defines an estimator for $\eta_i$ based on only $y_i$, the above inequality reduces to
\begin{equation}\label{m0}\begin{split}
\sup_{\beta\in\mB(n,s_n)}E_{\beta} \sum_{i=1}^n \widehat\eta_i1\{\eta_i=0\}\geq&
\sum_i\int_{\beta_i}\int_{y_i}\widetilde \eta_i(y_i)\phi(y_i;\beta_i,1)dy_i1\{\eta_i=0\}\pi_1(\beta_i)d\beta_i\\
 \geq&\sum_{i}(1-(s_n-1)/n)\int \widetilde\eta_i(y)\phi(y;0,1)dy,
\end{split}\end{equation}
where we use the fact $E_{\pi_1}\eta_i=(s_n-1)/n$.

On the other hand, for any $\widehat\beta\in\Omega_f=\{\widehat\beta: \sup_{\beta\in\mB(n,s_n)}E_{\beta}\|\beta-\widehat \beta\|^2\leq s_nf(n,s_n)\}$, $s_nf (n,s_n) \geq E_\beta\|\widehat\beta-\beta\|^2 \geq \sum_{i=1}^n\beta_i^2 E_\beta(1-\widehat\eta_i) $ for $\beta\in\mB(n,s_n)$,
then if all the nonzero coefficients of $\beta$ are $\{c_2f(n,s_n)\}^{1/2}$, we must have
$(1/c_2)s_n\geq \sum \eta_iE_{\beta}(1-\widehat\eta_i)=\sum \eta_i-\sum E_{\beta}\widehat\eta_i1\{\eta_i=1\}$. If we further take expectation $E_{\pi_2}$ on both side,
where $\pi_2$ is a uniform distribution
over $\mB_2=\{\beta\in \BR^n: \beta_i=0\mbox{ or }\sqrt{c_2f(n,s_n)}, \|\beta\|_0=s_n\}$,
then $(1/c_2)s_n\geq E_{\pi}(\sum \eta_i-\sum E_{\beta_2}\widehat\eta_i1\{\eta_i=1\})$.
Then by similar argument, we can derive that
\begin{equation}\label{m1}\begin{split}
&(1/c_2)s_n\geq s_n-E_{\pi}(\sum E_{\beta}\widehat\eta_i1\{\eta_i=1\})\\
=&s_n-\sum_{i=1}^n\int_{\beta_{i}}\int_{y_i}\phi(y_i;\beta_i,1)\int_{\beta_{-i}}\int_{y_{-i}}\widehat\eta_i(y)\prod_{j\neq i}^n\phi(y_j;\beta_j,1)dy_{-i}
\widetilde\pi_2(\beta_{-i})d\beta_{-i}dy_i1\{\eta_i=1\}\pi_1(\beta_i) d\beta_{i}\\
=&s_n-\sum_i(s_n/n)\int \widetilde\eta_i(y)\phi(y; \{c_2f(n,s_n)\}^{1/2},1)dy,
\end{split}\end{equation}
where 
the second equality holds since
$\widetilde\pi_2(\beta_{-i})=\pi_2(\beta_{-1}|\beta_i=1)=\pi_1(\beta_{-i}|\beta_i=0)=\widetilde\pi_1(\beta_{-i}).$

By Lemma \ref{nplemma}, subject to $\int \widetilde\eta_i(x)\phi(x; \{c_2f(n,s_n)\}^{1/2},1)\geq a_i$ for some $a_i\in[0,1]$ and $\widetilde\eta_i(\cdot)\in[0,1]$,
the following minimization holds: $\int \widetilde\eta_i(x)\phi(x; 0,1)\geq \Phi(\Phi^{-1}(a_i)-\{c_2f(n,s_n)\}^{1/2})$.
Hence subject to (\ref{m1}) (i.e., subject to $\int \widetilde\eta_i(y)\phi(y; \{c_2f(n,s_n)\}^{1/2},1)dy=a_i$ and 
$\sum a_i\geq n-n/c_2$), we have
\begin{equation}\label{m2}\begin{split}
&\sum_{i}\int \widetilde\eta_i(x)\phi(x;0,1)dx\geq \min_{\sum a_i =n- n/c_2}\sum_i\Phi(\Phi^{-1}(a_i)-\{c_2f(n,s_n)\}^{1/2})\\
\geq& n\Phi(\Phi^{-1}(1-1/c_2)-\{c_2f(n,s_n)\}^{1/2}),
\end{split}
\end{equation}
where the second inequality is due to Jensen's inequality.

Combine (\ref{m0}), (\ref{m1}) and (\ref{m2}), we have that for any $c_2>1$,
\[
\inf_{\widehat\beta\in\Omega_f}\sup_{\beta\in\mB(n,s_n)}E_{\beta} \sum_{i=1}^n \widehat\eta_i1\{\eta_i=0\}\geq (n-s_n+1)\Phi(\Phi^{-1}(1-1/c_2)-\{c_2f(n,s_n)\}^{1/2}),
\]
or equivalently, 
\begin{equation}\label{m3}
\inf_{\widehat\beta\in\Omega_f}\sup_{\beta\in\mB(n,s_n)}E_{\beta} \sum_{i=1}^n \widehat\eta_i1\{\eta_i=0\}\geq (n-s_n+1)\sup_{c_2>1}\Phi(\Phi^{-1}(1-1/c_2)-\{c_2f(n,s_n)\}^{1/2}).
\end{equation}

Asymptotically, 
\begin{enumerate}
    \item if $f(n,s_n)$ is of constant order, i.e. $\lim\sup f(n,s_n)=\kappa<\infty$, $\lim\sup s_n/n=\zeta<1$, then
    \[\begin{split}
    &\inf_{\widehat\beta\in\Omega_f}\sup_{\beta\in\mB(n,s_n)}E_{\beta} \sum_{i=1}^n \widehat\eta_i1\{\eta_i=0\}\geq H(\kappa)(n-s_n+1)\asymp H(\kappa)(1-\zeta)n,\end{split}\]
    where $H(\kappa)=\sup_{c_2>1}\Phi(\Phi^{-1}(1-1/c_2)-\{c_2\kappa\}^{1/2})$.
    \item if $f(n,s_n)$ diverges to infinity, i.e. $f(n,s_n)\rightarrow \infty$, then for any fixed $1<c_3<c_2$,
     \[\begin{split}
    &\inf_{\widehat\beta\in\Omega_f}\sup_{\beta\in\mB(n,s_n)}E_{\beta} \sum_{i=1}^n \widehat\eta_i1\{\eta_i=0\}\geq (n-s_n+1)\Phi(\Phi^{-1}[1-1/c_2]-\{c_2f(n,s_n)\}^{1/2})\\
    \geq& C(n-s_n+1)\exp\{-c_3f(n,s_n)/2\}\sim C(1-\zeta)
    n\exp\{-c_3f(n,s_n)/2\}, \end{split}\]
    for some absolute constant $C$,  where we use the inequality $\Phi(-x)\geq (1/\{2\pi\}^{1/2})\exp\{-x^2/2\}[x/(x^2+1)]$ for any $x>0$. Since $c_2$ can be arbitrarily closed 1,
    $\inf_{\widehat\beta\in\Omega_f}\sup_{\beta\in\mB(n,s_n)}E_{\beta} \sum_{i=1}^n \widehat\eta_i1\{\eta_i=0\}\geq  C(1-\zeta)n\exp\{-[1+o^+(1)]f(n,s_n)/2\}$.
\end{enumerate}
This concludes the proof.
\end{proof}

{\bf \noindent Proof of inequality (2.2)}
\begin{proof}
It is easy to see that
\begin{equation}\label{eqr1}
\begin{split}
 &\sup_{\beta\in\mB(n,s_n)}E\|\widehat\beta-\beta\|^2\\
&\leq  
\sum_{i=1}^n[Ey_i^21(\beta_i=0, |y_i|\geq \sqrt{(1-\epsilon_n) f(n,s_n)})+E(y_i-\beta_i)^21(\beta_i\neq0)\\
&\qquad+E\beta_i^21(
\beta_i\neq0, |y_i|\leq \sqrt{(1-\epsilon_n) f(n,s_n)})]\\
&< 2n\int_{\sqrt{ (1-\epsilon_n) f(n,s_n)}}^{\infty} x^2\phi(x)dx+s_n+s_n\sup_{\beta_i>0}\beta_i^2\Phi[\sqrt{ (1-\epsilon_n) f(n,s_n)}-\beta_i].\\
\end{split}
\end{equation}
For the first term on the RHS of \eqref{eqr1},
\[
\begin{split}
    &\int_{\sqrt{ (1-\epsilon_n) f(n,s_n)}}^{\infty} x^2\phi(x)dx=\frac{1}{\sqrt{2\pi}} \int_{ (1-\epsilon_n) f(n,s_n)/2}^{\infty} \exp(-t+\log(\sqrt{2t}))dt\\
   \leq  &\frac{1}{\sqrt{2\pi}} \int_{ (1-\epsilon_n) f(n,s_n)/2}^{\infty} \exp(-(1-\eta_n)t)dt=
   \frac{1}{\sqrt{2\pi}(1-\eta_n)}\exp[-(1-\eta_n)(1-\epsilon_n) f(n,s_n)/2],
\end{split}
\]
where $\eta_n=2\log(\sqrt{(1-\epsilon_n) f(n,s_n)})/(1-\epsilon_n) f(n,s_n)\rightarrow 0$.
Since $\lim\inf_n (1-\eta_n)(1-\epsilon_n) f(n,s_n)/2\log(n/s_n)\geq1$, we must have 
\begin{equation}\label{eqr2}
\begin{split}
    2n \int_{\sqrt{ (1-\epsilon_n) f(n,s_n)}}^{\infty} x^2\phi(x)dx \lesssim \frac{2}{\sqrt{2\pi}}s_n = O(s_n).
\end{split}
\end{equation}

For the term $\sup_{\beta_i>0}\beta_i^2\Phi[\sqrt{ (1-\epsilon_n) f(n,s_n)}-\beta_i]$, we note that 
\begin{equation}
    \begin{split}\label{eqr3}
        &\sup_{0<\beta_i\leq\sqrt{(1-\epsilon_n) f(n,s_n)}}\beta_i^2\Phi[\sqrt{ (1-\epsilon_n) f(n,s_n)}-\beta_i]\leq (\sqrt{ (1-\epsilon_n) f(n,s_n)})^2;\\
        &\sup_{\sqrt{(1-\epsilon_n) f(n,s_n)}<\beta_i\leq\sqrt{2(1-\epsilon_n) f(n,s_n)}}\beta_i^2\Phi[\sqrt{ (1-\epsilon_n) f(n,s_n)}-\beta_i]\leq (\sqrt{2 (1-\epsilon_n) f(n,s_n)})^2\times (1/2);\\
        &\beta_i^2\Phi[\sqrt{ (1-\epsilon_n) f(n,s_n)}-\beta_i] \mbox{ is decreasing when }\beta_i\in[\sqrt{2(1-\epsilon_n) f(n,s_n)}, \infty),
    \end{split}
\end{equation}
where the last fact is due to the exponential decay of normal tail probability.
    
Combine \eqref{eqr1}, \eqref{eqr2}, \eqref{eqr3} and the fact that $\epsilon_n\succ f(n,s_n)^{-1}$ we have that
\[
\sup_{\beta\in\mB(n,s_n)}E\|\widehat\beta-\beta\|^2\lesssim O(s_n)+ s_n + (1-\epsilon_n)s_nf(n,s_n)< s_nf(n,s_n), \mbox{ as } n\rightarrow \infty.
\]

\end{proof}

{\bf \noindent Proof of Theorem 2.3}

\begin{proof}
If $s_n/n\rightarrow \zeta >0$, the right-hand side of (2.5) asymptotically reduces to constant, hence the result is trivial. Therefore, we only consider the case that $s_n/n\rightarrow 0$.

Let's first restate the proof of non-asymptotic result of Theorem 2.1: If an estimator $\widehat\beta$ satisfies that $E(\mbox{FN}(\widehat\beta))\leq s_n/c_2$ when the true $\beta\in\mB(n,s_n)$ and its nonzero elements are $\{c_2f(n,s_n)\}^{1/2}$ with $c_2>1$, then we must have that $\sup_{\beta\in\mB(n,s_n)}E(\mbox{FP}(\widehat\beta))\geq (n-s_n+1)\Phi(\Phi^{-1}(1-1/c_2)-\{c_2f(n,s_n)\}^{1/2})$. This result will be applied later.

Given any $\widehat\beta\in\Omega_o(c_1)$, then when the true $\beta\in\mB(n,s_n)$ and its nonzero elements are $\{c_2c_1\log(n/s_n)\}^{1/2}$, we must have $E(\mbox{FN}(\widehat\beta))\leq s_n/c_2$. Now we consider a modified estimator $\widehat\beta^{(T)}$, such that $\widehat\beta^{(T)}_i=\widehat\beta_i1(|y_i|\leq t)$ for some $t =\{k\log(n/s_n)\}^{1/2}$ where $k$ will be specified later. 
When the true $\beta\in\mB(n,s_n)$ and its nonzero elements are $\{c_2c_1\log(n/s_n)\}^{1/2}$, we have that 
\[\begin{split}
    E(\mbox{FN}(\widehat\beta^{(T)}))&\leq E(\mbox{FN}(\widehat\beta))+s_n\times\int_{-t}^t\phi(x;\{c_2c_1\log(n/s_n)\}^{1/2},1)dx\\
    &\leq s_n\left[\frac{1}{c_2}+\int_{-t}^t\phi(x;\{c_2c_1\log(n/s_n)\}^{1/2},1)dx\right].
\end{split}
\]
We let $k=c_1c_2-\delta_n$ for some $\delta_n=o^+(1)$, such that $\int_{-t}^t\phi(x;\{c_2c_1\log(n/s_n)\}^{1/2},1)dx=\delta'_n=o^+(1)$ (Note that such $\delta_n$ exists due to $n/s_n\rightarrow\infty$), therefore, $E(\mbox{FN}(\widehat\beta^{(T)}))\leq s_n(1/c_2+\delta_n')$.

Now we apply the restated result of Theorem 2.1, and can obtain that
$\sup_{\beta\in\mB(n,s_n)}E(\mbox{FP}(\widehat\beta^{(T)}))\geq (n-s_n+1)\Phi(\Phi^{-1}(1-1/(c_2-\delta_n''))-c_1c_2\log(n/s_n)\}^{1/2})$
for some $\delta_n''=o^+(1)$. And for any $D$, 
\[\begin{split}
    &E\left(\frac{\mbox{FP}(\widehat\beta^{(T)})}{\mbox{FP}(\widehat\beta^{(T)})+s_n}\right)\geq E\left(\frac{\mbox{FP}(\widehat\beta^{(T)})}{\mbox{FP}(\widehat\beta^{(T)})+s_n}\right)1(\mbox{FP}(\widehat\beta^{(T)})\leq D)\\
    \geq & \frac{1}{D+s_n}
    {E(\mbox{FP}(\widehat\beta^{(T))}1(\mbox{FP}(\widehat\beta^{(T)})\leq D))}=\frac{1}{D+s_n}{[E(\mbox{FP}(\widehat\beta^{(T)}))-E(\mbox{FP}(\widehat\beta^{(T)})1(\mbox{FP}(\widehat\beta^{(T)})\geq D))]}.
\end{split}
\]
On the other hand, let $\tilde\beta$ be the hard-thresholding estimator with threshold value $t$, then
$\mbox{FP}(\widehat\beta^{(T)})\leq \mbox{FP}(\tilde\beta)$ and $E(\mbox{FP}(\widehat\beta^{(T)})1(\mbox{FP}(\widehat\beta^{(T)})\geq D))\leq E(\mbox{FP}(\tilde\beta)1(\mbox{FP}(\tilde\beta)\geq D))$.
Since $\mbox{FP}(\tilde\beta)\sim Bin(n-\|\beta\|_0, \mu_t)$, where $\mu_t=2\int_t^{\infty}\phi(x;0,1)dx=o(s_n/n)$, and let $D=s_n$, then we obtain that
\[\begin{split}
  &  E(\mbox{FP}(\tilde\beta)1(\mbox{FP}(\tilde\beta)\geq D))= (n-\|\beta\|_0)\mu_t Pr(Bin(n-1-\|\beta\|_0,\mu_t)\geq s_n-1)\\
  \leq  & n\mu_t \exp\{-s_n\log(s_n/n\mu_t)\},
\end{split}
\]
where the last inequality is due to Chernoff-Hoeffding bound of Binomial distribution.
Hence
\[\begin{split}
    &E\left(\frac{\mbox{FP}(\widehat\beta^{(T)})}{\mbox{FP}(\widehat\beta^{(T)})+s_n}\right)\geq \frac{1}{2s_n}
    {[E(\mbox{FP}(\widehat\beta^{(T)}))-n\mu_t \exp\{-s_n\log(s_n/n\mu_t)]}.
\end{split}
\]

At the end, we have
\begin{equation}\label{fdrint}
\begin{split}
   & \inf_{\widehat\beta\in\Omega_o(c_1)} \sup_{\beta\in\mB(n,s_n)}\mbox{FDR}(\widehat\beta)
     \geq \inf_{\widehat\beta\in\Omega_o(c_1)} \sup_{\beta\in\mB(n,s_n)}E\left(\frac{\mbox{FP}(\widehat\beta)}{\mbox{FP}(\widehat\beta)+s_n}\right)\\
    \geq & \inf_{\widehat\beta\in\Omega_o(c_1)} \sup_{\beta\in\mB(n,s_n)}E\left(\frac{\mbox{FP}(\widehat\beta^{(T)})}{\mbox{FP}(\widehat\beta^{(T)})+s_n}\right)\\
    \geq &\inf_{\widehat\beta\in\Omega_o(c_1)} \sup_{\beta\in\mB(n,s_n)}\frac{1}{2s_n}
    {[E(\mbox{FP}(\widehat\beta^{(T)}))-n\mu_t \exp\{-s_n\log(s_n/n\mu_t)]}\\
    \geq &\frac{1}{2s_n}
    {[(n-s_n+1)\Phi(\Phi^{-1}(1-1/(c_2-\delta_n''))-c_1c_2\log(n/s_n)\}^{1/2})-n\mu_t \exp\{-s_n\log(s_n/n\mu_t)]}.
\end{split}
\end{equation}

Taking supremum over $c_2$, and asymptotically let $n\rightarrow \infty$, we can obtains that 
\[ 
\inf_{\widehat\beta\in\Omega_o(c_1)} \sup_{\beta\in\mB(n,s_n)}\mbox{FDR}(\widehat\beta)
     \geq C\left(\frac{s_n}{n}\right)^{(c_1-2+o^+(1))/2}.
\]
with some $o(1)$ term, where $C$ is an absolute constant.

\end{proof}

\begin{lemma}[\cite{Baraud2002} or Lemma 5 of \cite{cai2014two}]\label{power}
 Consider any class of distributions $P_\theta$ with parameter space $\theta\in\mF$.
 Let $\mu_\rho$ be some probability measure on $\mF_\rho=\{\theta\in\mF, \|\theta-\theta_0\|\geq \rho\}$ and define measure 
 $P_{\mu_\rho}=\int P_\theta d\mu_\rho(\theta)$. Assuming that $P_{\mu_\rho}$ is absolutely continuous with respect to $P_0$,
 we define $L_{\mu_{\rho}}(y)=dP_{\mu_\rho}(y)/dP_0$. For all $\alpha>0$, $\nu\in[0, 1-\alpha]$, if $E_{\theta_0}\{L^2_{\mu_{\rho^*}}(Y)\}
 \leq 1+4(1-\alpha-\nu)^2$ for some $\rho^*>0$ and some distribution $\mu_{\rho^*}$, then
 \[
 \forall 0\leq \rho\leq \rho^*,\quad \inf_{\Phi_\alpha}\sup_{\theta\in\mF_\rho} P_\theta(\Phi_\alpha=0)\geq \nu, 
 \]
 where $\Phi_\alpha$ is the set of $\alpha$-level test, such that $P_{\theta_0}(\Phi_\alpha=0)=1-\alpha$.
\end{lemma}
Lemma \ref{power} studies the relationship between effect size $\rho$ and power $1-P_\theta(\Phi_\alpha = 0)$ for the hypothesis testing problem $H_0: \theta=\theta_0$. This lemma is useful for proving Theorem 2.5.

\vskip 0.1in
{\noindent \bf Proof of Theorem 2.5}
\begin{proof} 
For any estimator $\widehat\beta$, we assume its convergence rate follows
$\sup_{\beta\in\mB(n,s_n)}E\|\widehat\beta-\beta\|^2 =\lambda s_n\log n$ for some $\lambda$. Note here
$\lambda$ is not necessarily a constant, it can be a function of $n$ and $s_n$.

Let $\mB'(n,s_n)$ be the set of all $n$-dimensional vectors who have exactly $s_n$ nonzero entries, 
and all the nonzero entries are $\{c_3\log n\}^{1/2}$, for some constant $c_3$. 
Hence for any $\beta\in\mB'(n,s_n)$, $\lambda s_n\log n\geq E\|\widehat\beta-\beta\|^2\geq 
c_3\log n\sum_{i\in\xi(\beta)}P_{\beta}(\widehat\beta_i=0)$. This implies that
for any $\beta\in\mB'(n,s_n)$,
$\sum_{i\in\xi(\beta)}P_{\beta}(\widehat\beta_i=0)\leq \lambda s_n/c_3$, and there exist $i_1,\dots i_{s'}\in\xi(\beta)$ such that
$P_{\beta}(\widehat\beta_{i_j}=0)\leq 2\lambda/c_3$, where $s'=s_n/2$.

Note there are totally ${n \choose s_n}$ elements in $\mB'(n,s_n)$, each $\beta$ in $\mB'(n,s_n)$ has its own indices $i_1,\dots i_{s'}\in\xi(\bbeta)$ such that $P_{\beta}(\widehat\beta_{i_j}=0)\leq 2\lambda/c_3$ holds. 
This relationship can be viewed as one-to-$s'$ mapping from $\mB'(n,s_n)$ to $\mB'(n,s_n-1)$, i.e., from $\beta$ to $\{\beta^{(i_j)} \}_{j=1}^{s'}$ where $\beta^{(i_j)}$ is the vector which replaces $i_j$th entry of $\beta$ with 0.
Since there are totally ${n \choose s_n-1}$ elements in $\mB'(n,s_n-1)$, there must exists at least one $\beta_0\in\mB_{s_n-1}'$,
such that there are at least $m=s'{n \choose s_n}/{n \choose s_n-1} = (n-s_n+1)/2$ elements in $\mB'(n,s_n)$ that map to $\beta_0$.
These $m$ elements in $\mB'(n,s_n)$ are denoted by $\beta_0^{j}$ for $j=1,\dots,m$,  where $\beta_0^{j}$ and $\beta_0$ differ only at one entry (denoted by $k_j$th entry).

Now let $A = \{y: \xi(\beta_0) \supset\xi(\widehat\beta(y))\}$, 
the above argument implies that 
\begin{equation}\label{powww}
\max_{j\leq m}\int_A \phi(y;\beta_0^j, I)dy \leq 
2\lambda/c_3.
\end{equation}


Now, we implement Lemma \ref{power}.
Set $A$ is viewed as the accept region for hypothesis $H_0: \beta=\beta_0$, and the corresponding testing size 
$\alpha = P_{\beta_0}(A^c)\leq \sup_{\beta\in\mB(n,s_n)}P_\beta(\mbox{FP}>0)$.
Let $\mF_{\rho^*} = \{\beta_0^j, j=1,\dots, m\}$. Given a uniform $\mu_{\rho^*}$,
$P_{\mu_{\rho^*}} = (1/m)\sum_{j} P_{\phi(\cdot;\beta_0^j, I)}$, and then 
$L_{\mu_{\rho^*}}(y) = (1/m)\sum_{j=1}^m \exp\{t y_j-t^2/2\}$, where $t=\{c_3\log n\}^{1/2}$, and
\[
\begin{split}
 &E_{\beta_0}(L_{\mu_\rho^*}(y)^2) = (1/m^2)\left\{\sum_{i\neq j}E\exp\{t y_j+ty_i-t^2\} +\sum_{j}E\exp\{2t y_j-t^2\}\right\}\\
 =&(1/m^2)[m(m-1)\exp(0)+m\exp(t^2)]=1+(n^{c_3}-1)/m.
\end{split}
\]

Therefore, by lemma \ref{power}, when $c_3<1$, we have
\[
\max_j P_{\beta_0^j}(A)\geq 1-\alpha-\{(n^{c_3}-1)/m\}^{1/2}/2.
\]
Combining the above inequality with (\ref{powww}), we have that 
\[
\frac{\sup_{\beta\in\mB(n,s_n)}E\|\widehat\beta-\beta\|^2 }{s_n\log n} =\lambda\geq (c_3/2)(1-\sup_{\beta\in\mB(n,s_n)}P_\beta(\mbox{FP}>0)-\delta_n),\] 
where $ \delta_n=\{(n^{c_3}-1)/m\}^{1/2}/2\rightarrow 0$.

\end{proof}

{\noindent \bf Proof of Theorem 2.6}

\begin{proof}
Instead of considering rate-minimax estimator $\widehat\beta(X,y)$, we consider all rate-minimax estimating functions $\check\beta=(\check \beta_i(X,y,\beta_{-i}))_{i=1}^n$ (which is not necessarily an estimator, as $\check \beta_i$ could depend on true unknown parameter $\beta_{-i}$), and define the collection of rate-minimax estimating functions $\check\Omega_f^l:$
\[
\check\Omega_f^l=\{\check \beta: \sup_{\beta\in\mB(p_n,s_n)}E_XE_{\beta}\sum_i(\beta_i-\check \beta_i(X,y,\beta_{-i}))^2\leq s_nf(n,p_n,s_n)/n\}.
\]
Note that $\Omega_f^l\subset \check\Omega_f^l$, hence 
\[\inf_{\widehat\beta\in\Omega_f^l}\sup_{\beta\in\mB(p_n,s_n)}E_XE_{\beta} \sum_{i=1}^{p_n} \widehat\eta_i1\{\eta_i=0\}\geq \inf_{\check\beta\in\check\Omega_f^l}\sup_{\beta\in\mB(p_n,s_n)}E_XE_{\beta} \sum_{i=1}^{p_n} \check\eta_i1\{\eta_i=0\},\]
where $E_X$ denote expectation over the probability measure of $X$, $\check\eta_i$ is the selection function induced by $\check\beta_i$.
Note that $\check\eta_i(X,y,\beta_{-i})$ can be rewritten as  $\check\eta_i(z_i,X,\beta_{-i})$ where $z_i=y-X_{-i}\beta_{-i}$.

Define prior $\pi_1$ to uniform distribution on set $\mB_1=\{\beta\in \BR^{p_n}: \beta_i=0\mbox{ or }\sqrt{c_2f(n,p_n,s_n)/n}, \|\beta\|_0=s_n-1\}$ for some constant $c_2\geq 1$.
Let $E_{X_{-i}}$ denote expectation over the measure of $X$ but without the $i$th column,
$E_{x_i}$ denote expectation over the measure of $x_i$ which is the $i$th column of $X$,
$E_\pi$ denote the expectation over $\beta$ with respect to prior $\pi$, $E_y$ and $E_{z_i}$ denote the expectation over the condition probability of $y$ and $z_i$ respectively, given $X$ and $\beta$ (Note that  $E_y\check\eta_i(y,X,\beta_{-i})1\{\eta_i=0\}=E_{z_i}\check\eta_i(z_i,X,\beta_{-i})1\{\eta_i=0\}$.)
Similar to (\ref{m0}), we can derive that
\[\begin{split}
&\sup_{\beta\in\mB(p_n,s_n)}E_{X}E_{y} \sum_{i=1}^{p_n} \check\eta_i(y,X,\beta_{-i})1\{\eta_i=0\}\geq  E_{\pi_1} E_{X}E_{y}\sum_{i=1} \check\eta_i1\{\eta_i=0\}\\
=&E_{\pi_1} E_{X}\sum_{i}E_{z_i}\check\eta_i(z_i,X,\beta_{-i})1\{\eta_i=0\}\\
=&E_{X}\int_{\beta_i}\int_{\beta_{-i}}\sum_i\int_{z_i}\check\eta_i(z_i,X,\beta_{-i})g(z_i;\beta,X)dz_i\pi_1(\beta_{-i}|\beta_i)d\beta_{-i}1\{\beta_i=0\}\pi_1(\beta_i)d\beta_i\\
=&E_{X}\int_{\beta_i}\int_{\beta_{-i}}\sum_i\int_{z_i}\check\eta_i(z_i,X,\beta_{-i})\phi(z_i;0,I)dz_i\widetilde\pi_1(\beta_{-i})d\beta_{-i}1\{\beta_i=0\}\pi_1(\beta_i)d\beta_i\\
=&E_{x_i}\sum_i\int_{\beta_i}\int_{z_i}\widetilde \eta_i(z_i, x_i)1\{\beta_i=0\}\phi(z_i;0,I)dz_i\pi_1(\beta_i)d\beta_i\\
\end{split}\]
where $\widetilde \eta_i(z_i, x_i)=E_{X_{-i}}\int_{\beta_{-i}}\check\eta_i(z_i,X,\beta_{-i})\widetilde\pi_1(\beta_{-i})d\beta_{-i}$, $\widetilde\pi_1(\beta_{-i})=\pi_1(\beta_{-i}|\beta_i=0)$, and $g(z_i;\beta,X)$ denotes the density of $z_i$ given $\beta$ and $X$, which reduces to standard normal distribution when $\beta_i=0$.
Since $E_{\pi_1}\eta_i=(s_n-1)/p_n$, we concludes 
\begin{equation}\label{m6}
\begin{split}
&\sup_{\beta\in\mB(p_n,s_n)}E_{X}E_{y} \sum_{i=1}^{p_n} \check\eta_i(z_i,X,\beta_{-i})1\{\eta_i=0\}
\geq \sum_{i}E_{x_i}(1-(s_n-1)/p_n)\int \widetilde\eta_i(z,x_i)\phi(z;0,I_n)dz.
\end{split}\end{equation}

On the other hand, since $\check\eta_i$ is rate-minimax, similar to (\ref{m1}), we obtain that
\begin{equation}\label{m7}
\begin{split}
&(1/c_2)s_n\geq s_n-(s_n /p_n)\sum_iE_{x_i}\int \widetilde\eta_i(z,x_i)\phi(z,\{c_2f(n,p_n,s_n)/n\}^{1/2}x_i,I_n)dz.
\end{split}\end{equation}


By condition (C0), $P(x\in\mX_i=\{\|x_i\|\leq \{(1+c_3) n\}^{1/2} \})\geq P_0(c_3,n)$ for any $c_3>0$. 
We let $\mE_{\mX_i}$ denote the integral with respect to measure of $x_i$ truncated within $\mX_i$ (note this 
is not a probability measure, since its total measure is less than 1), and then from (\ref{m7}) we have  
\begin{equation}\label{m8}
\begin{split}
& (s_n/p_n)\sum_i\mE_{\mX_i}\int \widetilde\eta_i(z,x_i)\phi(z,\{c_2f(n,p_n,s_n)/n\}^{1/2}x_i,I_n)dz
\geq  s_n-(1/c_2)s_n-s_n\sum_{i}(1-P(\mX_i))/p_n.
\end{split}\end{equation}

Use the arguments of Lemma \ref{nplemma}, subject to $\int \widetilde\eta_i(z,x_i)\phi(z;\{c_2f(n,p_n,s_n)/n\}^{1/2}x_i,I_n)dz\geq a_i$ for a given $a_i\in[0,1]$ and $x_i\in\BR^n$,
the following holds:
$\int \widetilde\eta_i(z,x_i)\phi(z; 0,I_n)dz\geq \Phi(\Phi^{-1}(a_i)-\{c_2f(n,p_n,s_n)/n\}^{1/2}\|x_i\|)$, and thus respectively
on $\mX_i$, we must have $\int \widetilde\eta_i(z,x_i)\phi(z; 0,I_n)dz\geq \Phi(\Phi^{-1}(a_i)-\{(1+c_3)  c_2f(n,p_n,s_n)\}^{1/2})$.

Therefore,
\[\begin{split}
&\sum_{i}(1-(s_n-1)/p_n)E_{x_i}\int \widetilde\eta_i(z,x_i)\phi(z;0,I_n)dz
\geq (1-\frac{s_n-1}{p_n})\sum_{i}\mE_{\mX_i}\int \widetilde\eta_i(z,x_i)\phi(z;0,I_n)dz\\
\geq &(1-\frac{s_n-1}{p_n})\sum_{i}\mE_{\mX_i} \Phi\left(\Phi^{-1}\left[\int \widetilde\eta_i(z,x_i)\phi(z;\{c_2f(n,p_n,s_n)/n\}^{1/2}x_i,I_n)dz\right]-\{(1+c_3) c_2 f(n,p_n,s_n)\}^{1/2}\right)\\
\geq &(1-\frac{s_n-1}{p_n}) \sum_iP(\mX_i)\\
\times&\Phi\left(\Phi^{-1}\left(\frac{\mE_{\mX_i}\int \widetilde\eta_i(z,x_i)\phi(z;\{c_2f(n,p_n,s_n)/n\}^{1/2}x_i,I_n)dz}{P(\mX_i)}\right)-\{(1+c_3) c_2 f(n,p_n,s_n)\}^{1/2}\right)\\
\geq &(1-\frac{s_n-1}{p_n})p_n P_0\\
\times&\Phi\left(\Phi^{-1}\left(\sum_i\frac{\mE_{\mX_i}\int \widetilde\eta_i(z,x_i)\phi(z;\{c_2f(n,p_n,s_n)/n\}^{1/2}x_i,I_n)dz}{p_n}\right)-\{(1+c_3)c_2 f(n,p_n,s_n)\}^{1/2}\right)\\
\geq &(1-\frac{s_n-1}{p_n}) P_0p_n\Phi\left(\Phi^{-1}\left\{{1-(1/c_2)-(1-P_0)}\right\}-\{(1+c_3) c_2 f(n,p_n,s_n)\}^{1/2}\right).
\end{split}\]
where $P_0=P_0(c_3,n)$, and the third and fourth inequalities are due to Jensen's inequality, and the last inequality is due to (\ref{m8}).

Combining the above result with (\ref{m6}), we have that 
\begin{equation}\label{m9}
\begin{split}
&\inf_{\widehat\beta\in\Omega_f^l}\sup_{\beta\in\mB(n,s_n)}E_XE_{\beta} \sum_{i=1}^n \widehat\eta_i1\{\eta_i=0\}\geq (p_n-s_n+1)\sup_{c_2>1,c_3>0}P_0(c_3,n)\\
&\times\Phi\left(\Phi^{-1}\left({P_0(c_3,n)-1/c_2}\right)-\{(1+c_3)c_2f(n,p_n,s_n)\}^{1/2}\right).
\end{split}
\end{equation}

Asymptotically, 
\begin{enumerate}
    \item if $f(n,p_n,s_n)$ is of constant order, i.e. $\lim\sup f(n,p_n, s_n)=\kappa<\infty$, $\lim\sup s_n/p_n=\zeta<1$, then when $n$ is sufficiently large,
    \[\begin{split}
    &\inf_{\widehat\beta\in\Omega_f^l}\sup_{\beta\in\mB(n,s_n)}E_XE_{\beta} \sum_{i=1}^n \widehat\eta_i1\{\eta_i=0\}\gtrsim H(\kappa)(p_n-s_n+1)\asymp H(\kappa)(1-\zeta)p_n,\end{split}\]
    where $H(\kappa)=\sup_{c_2>1}\Phi(\Phi^{-1}(1-1/c_2)-\{c_2\kappa\}^{1/2})$.
    \item if $f(n,p_n,s_n)$ diverges to infinity, i.e. $f(n,p_n,s_n)\rightarrow \infty$, then for any fixed $1<c_4<\lambda c_2$, if $n$ is sufficiently large,
     \[\begin{split}
    &\inf_{\widehat\beta\in\Omega_f^l}\sup_{\beta\in\mB(n,s_n)}E_XE_{\beta} \sum_{i=1}^n \widehat\eta_i1\{\eta_i=0\}\gtrsim (p_n-s_n+1)\Phi(\Phi^{-1}[1-1/c_2]-\{\lambda c_2f(n,s_n)\}^{1/2})\\
    \geq& C(p_n-s_n+1)\exp\{-c_4f(n,p_n,s_n)/2\}\sim C(1-\zeta)
    p_n\exp\{-c_4f(n,s_n)/2\}, \end{split}\]
    for some absolute constant $C$. Since, $\lambda c_2$ can be arbitrarily closed to 1, we 
    $$\inf_{\widehat\beta\in\Omega_f^l}\sup_{\beta\in\mB(n,s_n)}E_{\beta} \sum_{i=1}^n \widehat\eta_i1\{\eta_i=0\}\geq  C(1-\zeta)n\exp\{-[1+o(1)]f(n,p_n,s_n)/2\}.$$
\end{enumerate}
This concludes the proof.

\end{proof}

{\bf \noindent Proof of Theorem 2.8}
\begin{proof}
The proof of Theorem 2.8 is similar to the proof of Theorem 2.3, hence only a sketch of proof without details is provided. 

We only consider the case that $p_n/s_n\rightarrow\infty$ since the case of $p_n/s_n\rightarrow 0$ is trivial. For any estimation function $\check\beta=(\check \beta_i(z_i,y,\beta_{-i}))_{i=1}^n$ that possesses convergence rate $\{c_1\log(p_n/s_n)/n\}^{1/2}$, we must have that: when the true $\beta\in\mB(n,s_n)$ and its nonzero elements are $\{c_2c_1\log(p_n/s_n)/n\}^{1/2}$, we must have $E(\mbox{FN}(\check\beta))\leq s_n/c_2$.
We construct a new estimation function $\check\beta^{(T)}$ based on $\check\beta$ as $\check\beta^{(T)}_i=\check\beta_i1(|\langle z_i,x_i/\|x_i\|\rangle|\geq\|x_i\|\sqrt{c_1c_2\log(p_n/s_n)/n}-o(1)\sqrt{\log(p_n/s_n)})$ for some $o(1)$ term, such that 
$E(\mbox{FN}(\check\beta))\leq s_n/c_2+s_n\times o(1)$.
The rest proof follows the same routine as in (\ref{fdrint}).
\end{proof}

{\noindent \bf Proof of Theorem 2.10}
\begin{proof}
This proof is similar to the proof of Theorem 2.5.

Assume an estimator $\widehat\beta$ satisfies $\sup_{\beta\in\mB(p_n,s_n)}E_XE_\beta\|\widehat\beta-\beta\|^2 =\lambda s_n\log p_n/n$ for some $\lambda$,
where $\lambda$ can be a function of $n$, $s_n$ and $p_n$.

Let $\mB'(p_n,s_n)$ be the set of all $p_n$-vectors who have exactly $s_n$ nonzero entries, and all the nonzero
entries have value $\{c_3\log p_n/n\}^{1/2}$ for some small constant $c_3$. 
By the same argument used in the proof of theorem 2.5, there exists a $\beta_0\in\mB'(p_n,s_n-1)$,
and $\beta_0^j\in \mB'(p_n,s_n)$ for $j=1,\dots,m=(p-s+1)/2$, such that $\beta_0^j$ and $\beta_0$ differ by only one entry and 
\begin{equation}\label{poww}
\max_{j\leq m}\int_A g(X,y;\beta_0^j)dXdy \leq 
2\lambda/c_3.
\end{equation}
where $A = \{(X,y): \xi(\beta_0) \supset\xi(\widehat\beta(X,y))\}$, and $g(\cdot,\cdot;\beta)$ is the data generation density function given regression parameter $\beta$.


Similar, we view $A$ as the accept region for hypothesis $H_0: \beta=\beta_0$, and the corresponding testing size 
$\alpha = P_{\beta_0}(A^c)\leq \sup_{\beta\in\mB(p_n,s_n)}P_\beta(\mbox{FP}>0)$.
Let $\mF_{\rho^*} = \{\beta_0^j, j=1,\dots, m\}$. Given a uniform $\mu_{\rho^*}$,
$P_{\mu_\rho^*}(X,y) = (1/m)\sum_{j} g(X,y;\beta_0^j)$ and  
$L_{\mu_\rho^*}(X, y) = P_{\mu_\rho^*}(X,y)/g(X,y;\beta_0)$.
Thus, we can show that
\[
\begin{split}
 &E_{\beta_0}(L_{\mu_\rho^*}(X,y)^2) = \frac{m(m-1)(1-c_2^2\log^2 p_n/n^2)^{-n/2}+m(1-4c_2\log p_n/n)^{-n/2} }{m^2}.
\end{split}
\]
If $\log p_n\prec n^{-1/2}$ and we let $8c_2<1$, $E_{\beta_0}(L_{\mu_\rho^*}(X,y)^2)=[(m-1)(1+O(\log^2 p_n/n))+O(\exp\{8c_2\log p_n\})]/m=1+o(1)$.

Therefore, by Lemma \ref{power}, we have
\[2\lambda/c_3\geq 1-\sup_{\beta\in\mB(p_n,s_n)}P_\beta(\mbox{FP}>0)-\{E_0(L_{\mu_\rho^*}(X,y)^2) -1\}^{1/2}/2=1-\sup_{\beta\in\mB(p_n,s_n)}P_\beta(\mbox{FP}>0)-o(1).\]
\end{proof}

\begin{lemma}\label{chi}
 Let $\chi^2_{d}(\kappa)$ be a chi-square distribution with degree of freedom $d$, and noncentral parameter $\kappa$, then we have the following concentration inequality
\[Pr( \chi^2_{d}(\kappa)>d+\kappa+2x+\{(4d+8\kappa)x\}^{1/2}\leq \exp(-x), \mbox{ and}\] 
\[Pr( \chi^2_{d}(\kappa)<d+\kappa-\{(4d+8\kappa)x\}^{1/2}\leq \exp(-x).\]
\end{lemma}
The proof follows the argument of lemma 1 of \cite{LaurentM2000}.

\vskip 0.1in
Now let us state a condition on the design matrix $X$, which will be used in the following Lemmas.
\begin{itemize}
 \item[C2:] There exist two constants $\underline\lambda$, $\overline\lambda$, 
such that all the eigenvalues of $X^T_\xi X_\xi/n$ are bouned within $(\underline\lambda,\overline\lambda)$ 
for any $|\xi|\leq q_n= M s_n$ for some sufficiently large $M$.
\end{itemize}

\begin{lemma}\label{SCthm}
  For a regression problem $y=X\beta^{(n)}+\varepsilon$ where $\beta^{(n)}\in\mB_0(s_n,p_n)$ and 
  $X$ satisfies condition (C2), if $\widetilde p_n$ satisfies either
1.  $s_n\leq\widetilde p_n\prec p_n$ and $\gamma>2$; or
2.  $\limsup s_n/p_n<\zeta$, $\widetilde p_n=\min(n,p_n)$ and $\gamma$ is satisfies (\ref{gamma3}),
 then the estimator (4.1) satisfies
\begin{align}
 &\lim_nP_{\beta^{(n)}}(\|\widehat\beta-\beta^{(n)}\|\leq C\{s_n\log(p_n/s_n)/n\}^{1/2})=1,\label{rate}\\
 &\lim_nP_{\beta^{(n)}}(\|\widehat\beta\|_0\leq C's_n)=1,\label{size}
\end{align}
for some sufficiently large $C$ and $C'$, where $C'$ can be arbitrarily close to 1 as long as $\gamma$ is sufficiently large. 
\end{lemma}

\begin{proof}
Let $\mbox{SC}(\xi)=\mbox{RSS}(\xi)+\gamma\log(p_n^{|\xi|}/|\xi|!)$ and $\xi^*=\xi(\beta^{(n)})$.
Thus $\widehat\beta=\beta_{\widehat\xi}^o$ where $\widehat\xi=\min_{|\xi|\leq \widetilde p_n}\mbox{SC}(\xi)$ and $\beta_{\xi}^o$ denotes the OLS estimation
under the model $\xi$.

First, we show the result (\ref{size}), i.e., the model size of $\widehat\beta$ is bounded. 
Let $\Xi=\{\xi; \mbox{SC}(\xi)\leq \mbox{SC}(\xi^*)\}$, then it is sufficient to show that $\Xi\subset\{\xi: |\xi|\leq C's_n\}$ for some $C'$ with large probability.
Note that 
\begin{equation}\label{scdiff}
\begin{split}
\mbox{SC}(\xi)-\mbox{SC}(\xi^*)&\geq (\mbox{RSS}(\xi\cup\xi^*)-\mbox{RSS}(\xi^*))+\gamma\log (p_n^{|\xi|-|\xi^*|}|\xi^*|!/|\xi|!),
\end{split}
\end{equation}
where $-(\mbox{RSS}(\xi\cup\xi^*)-\mbox{RSS}(\xi^*))=^d \chi^2_{|\xi\backslash\xi^*|}$. 
By Lemma \ref{chi}, uniformly for all $|\xi| = S$, with probability larger than $1-\exp\{-c_1C's_n\log(ep_n/C's_n)\}$,
\begin{equation}\label{deltasse}
\begin{split}
 &{-\mbox{RSS}(\xi\cup\xi^*)+\mbox{RSS}(\xi^*)}\\
 \leq&|\xi|+2[\log{p_n\choose S} +c_1C's_n\log(ep_n/C's_n)]+2\left\{ |\xi|[\log{p_n\choose S}+c_1C's_n\log(ep_n/C's_n)]\right\}^{1/2},
\end{split}
\end{equation}
for any $c_1>0$.

On the other hand, for $|\xi|\in (C's_n,p_n]$, by Stirling's approximation, $(2\pi)^{1/2}n^{n+1/2}e^{-n}e^{1/(12n+1)}<n!<(2\pi)^{1/2}n^{n+1/2}e^{-n}e^{1/(12n)}$,
the second term on the right handed side of (\ref{scdiff}) satisfies
\begin{equation}\label{deltap}
\begin{split}
 \gamma\log (p_n^{|\xi|-s_n}s_n!/|\xi|!) &\geq\gamma\log(p_n^S/S!)\left\{\frac{\lceil (C'-1)s_n\rceil\log (ep_n/s_n)-(\lceil C's_n\rceil+1/2)\log[\lceil C's_n\rceil/s_n] }{\lceil C's_n \rceil \log (ep_n/s_n)}\right\} 
\end{split}
\end{equation}
Note that if $p_n\succ s_n$, then the term inside the curly brackets on the right-hand side of (\ref{deltap}) converges to 
$1-1/C'$.

If $s_n\prec p_n$ and $C's_n\leq S\leq\widetilde p_n\prec p$, then $S=o(\log p_n^S/S!)$, $\log (p_n^S/S!)/\max(\log{p_n \choose S}, C's_n\log (ep_n/C's_n))\gtrsim1$.
Combined with (\ref{scdiff}), (\ref{deltasse}) and (\ref{deltap}), we have that if
\[\gamma \left(\frac{\lceil (C'-1)s_n\rceil\log (ep_n/s_n)-(\lceil C's_n\rceil+1/2)\log[(\lceil C's_n\rceil/s_n] }{\lceil C's_n \rceil \log (ep_n/s_n)}\right) >2+c_1,\]
then (\ref{deltap}) is asymptotically larger than (\ref{deltasse}), and (\ref{scdiff}) is positive. 
Therefore if $s_n\prec p_n$ and $\widetilde p_n\prec p_n$, for any $\gamma>2$, there exist a sufficiently large constant $C'$ and a sufficiently small $c_1$, such that (\ref{scdiff}) is positive for all $\widetilde p_n\geq |\xi|>C's_n$, with probability at least
$1-\widetilde p_n\exp\{-c_1C'\log(ep_n/s_n)\}=1-o(1)$ asymptotically.

When $C's_n\leq S\leq p_n $, by Stirling's approximation, 
$\log(p_n^S/S!)/\max\{ C's_n\log (ep_n/C's_n), S, \log{p_n\choose S} \} \gtrsim 1  $, 
hence, as long as
\begin{equation}\label{gamma2}
\gamma\left(\frac{\lceil (C'-1)s_n\rceil\log (ep_n/s_n)-(\lceil C's_n\rceil+\frac{1}{2})\log[(\lceil C's_n\rceil/s_n] }{\lceil C's_n \rceil \log (ep_n/s_n)}\right)>3+2c_1+2(1+c_1)^{1/2}, \mbox{ for all } C's_n\leq p_n,
\end{equation}
then (\ref{scdiff}) is positive for all $|\xi|>C's_n$ with probability at least
$1-p_n\exp\{-c_1C'\log(ep_n/{C's_n})\}$.
If $C'>1$ is an integer and $\limsup s_n/p_n\leq \zeta<1$, then the (\ref{gamma2}) reduces to 
\begin{equation}\label{gamma3}
\gamma\left(\frac{C'-1}{C'}-\frac{\log C'}{(1+\log(\max\{C',1/\zeta))\}}\right)>5, \mbox{ or }
\gamma\max_{C'\in\mathbb{N}^+}\left(\frac{C'-1}{C'}-\frac{\log C'}{(1+\log(\max\{C',1/\zeta))\}}\right)>5
\end{equation}
given small $c_1$. On the other hand, for any $1<C'<1/\zeta$, (even if $C'$ is very close to 1)
it is not difficult to see that
\[
\liminf \left(\frac{\lceil (C'-1)s_n\rceil\log (ep_n/s_n)-(\lceil C's_n\rceil+1/2)\log[(\lceil C's_n\rceil/s_n] }{\lceil C's_n \rceil \log (ep_n/s_n)}\right)>0,
\]
thus that (\ref{gamma2}) still holds if $\gamma$ is sufficiently large.  


Secondly, let us prove (\ref{rate}). It is sufficient to show that 
uniformly for any $\xi\in\Xi\cap\{\xi: |\xi|\leq C's_n\}$, $\|\beta_{\xi}^o - \beta^{(n)}\|^2\leq C{s_n\log(p_n/s_n)/n}$ for some constant $C$.

For any $\xi\supset\xi^*$ and $|\xi|\leq C's_n$, by condition \red{(C2)} and Lemma \ref{chi}, 
given a large $C$,
\begin{equation}\label{betadiff}
\begin{split}
 &P\{\|\beta_{\xi}^o-\beta^{(n)}\|^2\geq C{s_n\log(ep_n/s_n)/n}\}
\leq Pr\{\chi_{|\xi|}^2\geq \overline\lambda Cs_n\log(ep_n/s_n)\}
\leq \exp(-c_2s_n\log(ep_n/s_n))\\
\end{split}
\end{equation}
for some constant $c_2$.

For any $\xi\in\Xi$, $|\xi|\leq Cs$ and $\xi^*\nsubseteq\xi$, note that
\[
0>\mbox{SC}(\xi)-\mbox{SC}(\xi^*)>
\mbox{RSS}(\xi)-\mbox{RSS}(\xi\cup\xi^*)+\mbox{RSS}(\xi\cup\xi^*)-\mbox{RSS}(\xi^*)-\gamma \sum_{j=1}^{|\xi^*|}\log(p_n/j).
\]
By Lemma \ref{chi} and condition \red{(C2)}, we have $|\mbox{RSS}(\xi\cup\xi^*)-\mbox{RSS}(\xi^*)| = O(s_n\log(ep_n/s_n))$ with dominating probability,
thus $[\mbox{RSS}(\xi)-\mbox{RSS}(\xi\cup\xi^*)]=O(s_n\log(ep_n/s_n))$. Note that $[\mbox{RSS}(\xi)-\mbox{RSS}(\xi\cup\xi^*)]$ follows a noncentral 
chi-squared distribution, thus by Lemma \ref{chi}, we must have that
$\|\beta_{\xi^*\backslash\xi}^{(n)}\|^2=O({s_n\log(p_n/s_n)/n})$. 
This furthermore leads to that 
\begin{equation} \label{betadiff2}
\begin{split}
 &P\{\|\beta_{\xi}^{o}-\beta^{(n)}\|^2\geq C{s_n\log(ep_n/s_n)/n}\}
\leq \exp(-c_3s_n\log(ep_n/s_n)),
\end{split}
\end{equation}
for some positive $c_3$ given  a sufficiently large $C$. Combine (\ref{betadiff}) and (\ref{betadiff2}),
\[
\begin{split}
&P\{\|\beta_{\xi}^{o}-\beta^{(n)}\|^2\geq C{s_n\log(ep_n/s_n)/n},\mbox{ for all }\xi\in\Xi, |\xi|\leq C's_n\}\\
\leq& C's_n{p_n\choose C's_n}\exp(-\min(c_2,c_3)s_n\log(ep_n/s_n))\leq \exp(-c_4s_n\log(p_n/s_n)),
\end{split}
\]
where the last inequality holds if $c_2$ and $c_3$ is large enough which is ensured by a sufficiently large $C$.
\end{proof}

{\noindent \bf Proof of Theorem 4.1}
\begin{proof}
In this proof, we only show the results corresponding to the case $s_n\prec n$. The proof for the general case that $\limsup s_n/n\leq \zeta$ is similar, we igore it.

 First of all, the normal means model can be rewritten as $y=n^{1/2}(\beta^{(n)}/n^{1/2})$, i.e., we view $n^{1/2}I$ as the design matrix.
 Therefore, it satisfies condition C2, and use the same arguments in the proof of Lemma \ref{SCthm}, we have that $\|\widehat\beta-\beta^{(n)}\|^2\leq Cs_n\log(n/s_n)$
 and  $\|\widehat\beta\|_0\leq C's_n$, With dominating probability.
 Let $\widehat\xi=\xi(\widehat\beta)$, $\xi_1=\xi^*\cap\widehat\xi$ and $\xi_2=\widehat\xi\backslash\xi_1$, then 
with high probability, 
 \begin{equation}\label{l0}
\begin{split}
&\mbox{SC}(\xi_1)-\mbox{SC}(\xi_2\cup\xi_1)=\sum_{i\in\xi_2}\varepsilon_i^2-\gamma\sum_{i=|\xi_1|+1}^{|\xi_2\cup\xi_1|} \log (n/i)\leq (2+\delta)|\xi_2|\log (n/|\xi_2|)- \gamma\sum_{i=|\xi_1|+1}^{|\xi_2\cup\xi_1|} \log (n/i)\\
=&(2+\delta+o(1))\sum_{i=1}^{|\xi_2|}\log (n/i)- \gamma\sum_{i=1}^{|\xi_2|} \log (n/(|\xi_1|+i))
\end{split}
\end{equation}
where the inequality holds with dominating probability for any $\delta>0$ by Lemma \ref{chi}, and last equation  is due to $ \sum_{i=1}^{|\xi_2|}\log (n/i)\sim|\xi_2|\log (n/|\xi_2|)$.
Let $\gamma'$ be any constant satisfying $(1+\gamma')(2+\delta)<\gamma$.
We can show that (\ref{l0}) is negative if $|\xi_2|\geq \max\{1, |\xi_1|(|\xi_1|/n)^{\gamma'}\}$ when $n$ is sufficiently large.
This is due the following two facts.
First, if $C's_n\geq i\geq|\xi_1|$,
\[
\begin{split}
&(2+\delta+o(1))\log (n/i)-\gamma\log(n/(|\xi_1|+i))=\gamma \log((|\xi_1|+i)/i)-(\gamma-2-\delta-o(1))\log(n/i)\\
\leq &\gamma\log 2-(\gamma-2-\delta-o(1))\log(n/(C's_n))<0, \quad\,\mbox{when $n$ is large}.
\end{split}
\]
Second, if $|\xi_1|\geq i>|\xi_1|(|\xi_1|/n)^{\gamma'}$,
\[
\begin{split}
&(2+\delta+o(1))\log (n/i)-\gamma\log(n/(|\xi_1|+i))\\
=&(2+\delta+o(1))\log (1+|\xi_1|/i)-(\gamma-2-\delta-o(1))\log(n/(|\xi_1|+i)))\\
<&(2+\delta+o(1))\log [1+(n/|\xi_1|)^{\gamma'}]-(\gamma-2-\delta-o(1))\log(n/|\xi_1|)
<0, \quad\,\mbox{when $n$ is large}.
\end{split}
\]
Third, for $i=\lceil |\xi_1|(|\xi_1|/n)^{\gamma'}\rceil$,
\[
\begin{split}
 &(2+\delta)i\log (n/i)- \gamma\sum_{j=|\xi_1|+1}^{|\xi_1|+i} \log (n/i)\leq (2+\delta)i\log (n/i)- \gamma i\log(n/(|\xi_1|+i))\\
 \leq &i[(2+\delta)\log (n/i)- \gamma \log(n/(2|\xi_1|)) ]
 = i[(2+\delta)\log(2|\xi_1|/i)+(\gamma-2-\delta)\log(n/(2|\xi_1|))  ].
\end{split}
\]
Together, these above three inequalities imply our result on the false positive control of $\widehat\beta$.

Now we study the convergence rate of $\widehat\beta$. Note that
$\|\widehat\beta-\beta^{(n)}\|^2 = \sum_{i\in\xi_1}\varepsilon_i^2 + \sum_{i\in\xi_2}\varepsilon_i^2+\sum_{i\in\xi^*\backslash\xi_1}[\beta^{(n)}_i]^2$.
And $\sum_{i\in\xi_1}\varepsilon_i^2=O_p(s_n)=o_p(s_n\log(n/s_n))$, $ \sum_{i\in\xi_2}\varepsilon_i^2 = O_p(|\xi_2|\log(n/|\xi_2|))$
 by Lemma 6 of \cite{WuZ2013}, thus we only need to show that $\sum_{i\in\xi^*\backslash\xi_1}[\beta^{(n)}_i]^2\leq (\gamma+\epsilon)s_n\log (n/s_n)$
 with high probability, for any positive small $\epsilon$. Note that this can be derived from the facts that
$\mbox{SC}(\xi_1\cup\xi_2)<\mbox{SC}(\xi^*)$, and $|\xi_2|/s_n$ is sufficiently small.
\end{proof}

\begin{lemma}\label{diff}
For regression model under Gaussian random design (condition C1)
with $\beta^{(n)}\in\mB_0(s_n,p_n)$, $s_n\log (p_n/s_n)\prec n$ and $s_n\prec p_n$,
given any positive constants $C$, $M$ and tiny constant
$\epsilon>0$, we have that
\[ y^T(P_{\xi_1\cup\xi_2}-P_{\xi_1})y\leq\left[C_0 \{s_n^2\log^2(p_n/s_n)/n\}^2+\{(2+\epsilon)|\xi_2|\log(p/|\xi_2|)\}^{1/2}\right]^2\]
holds with high probability uniformly for all $\xi_1\subseteq\xi^*=\xi(\beta^{(n)})$,
$\|\beta_{\xi^*\backslash\xi_1}^{(n)}\|\leq M\{s_n\log(p_n/s_n)/n\}^{1/2}$,
 $0<|\xi_2|\leq Cs$ and $\xi_2\subset\xi^{*c}$ where $C_0$ is some positive constant and 
 $P_\xi$ denotes the projection matrix induced by $X_{\xi}$.
\end{lemma}

\begin{proof}
First of all, conditional on $X_{\xi^*}$, $\varepsilon^TP_{\xi^*}\varepsilon \sim \chi_s^2$,
thus by Lemma \ref{chi}, with probability at least $1-\exp\{-c_1s_n\log(p_n/s_n)\}$,
\begin{equation}\label{chi1}
\varepsilon^TP_{\xi^*}\varepsilon\leq s_n\log (p_n/s_n),
\end{equation}
for some positive $c_1$.
Similarly, given a subset model $\xi$, for a random vector $\widetilde\varepsilon\sim N(0, \sigma^2)$ that is independent of $X_\xi$, with probability $1-\exp((1+c_2)|\xi|\log(p_n/|\xi|))$,
\begin{equation}\label{chi2}
\widetilde\varepsilon^TP_{\xi}\widetilde\varepsilon\leq  |\xi|+(2+2c_2)|\xi|\log(p_n/|\xi|) +2|\xi|\{(1+c_2)\log(p_n/|\xi|)\}^{1/2},
\end{equation}
for any $c_2>0$.

By Corollary 5.35 of \cite{Vershynin2012}, 
given a sufficiently large $n$, with probability at least $1-\exp\{-c_3s_n\log(p_n/s_n)\}$,
uniformly for all $|\xi_1\cup\xi_2|\leq (C+1)s$,
the singular values of of $X_{\xi_1\cup\xi_2}/{n}^{1/2}$ are inside the  interval 
$(1-\delta_n,1+\delta_n)$, where $c_3$ is some positive constant and $\delta_n\asymp \{s_n\log(p_n/s_n)/n\}^{1/2}$.
This implies that $X_{\xi_1}$ and $X_{\xi_2}$ are almost orthogonal, and  by matrix algebra, we can show that
\[
\begin{split}
 &y^T(P_{\xi_1\cup\xi_2}-P_{\xi_1})y
= (y-X_{\xi_1}\beta^{(n)}_{\xi_1})^T(P_{\xi_1\cup\xi_2}-P_{\xi_1})(y-X_{\xi_1}\beta^{(n)}_{\xi_1})\\
=&(C'\delta_n\|P_{\xi_1} (y-X_{\xi_1}\beta^{(n)}_{\xi_1})\|+ (1+C'\delta_n)\|  P_{\xi_2}(y-X_{\xi_1}\beta^{(n)}_{\xi_1})\|)^2\\
\leq&(C'\delta_n\|P_{\xi^*} \varepsilon\|+C'\delta_n\|P_{\xi_1} X_{\xi^*\backslash\xi_1}\beta^{(n)}_{\xi^*\backslash\xi_1}\| +
(1+C'\delta_n)\| P_{\xi_2} (y-X_{\xi_1}\beta^{(n)}_{\xi_1})\|)^2\\
\leq&(C'\delta_n\|P_{\xi^*} \varepsilon\|+C''\delta_n^2(1+\delta_n) (n)^{1/2}\|\beta^{(n)}_{\xi^*  \backslash\xi_1}\| +
(1+C'\delta_n)\| P_{\xi_2} (y-X_{\xi_1}\beta^{(n)}_{\xi_1})\|)^2,
 \end{split}
\]
for some constants $C'$ and $C''$.

Note that $(y-X_{\xi_1}\beta^{(n)}_{\xi_1})\sim N(0, (1+\|\beta_{\xi^{(n)}\backslash\xi_1}\|^2)I_n)$
is independent to $X_{\xi_2}$. Thus, combining the above inequality with (\ref{chi1}) and (\ref{chi2}), 
we have that with probability at least
$1-\exp\{-c_1s_n\log(p_n/s_n)\}-\sum_{i=1}^{Cs_n}{p_n \choose i}\exp\{(1+c_2)i\log(p_n/i)\} -\exp\{-c_3s_n\log(p_n/s_n)\}(=1-o(1))$,
\[
\begin{split}
 &y^T(P_{\xi_1\cup\xi_2}-P_{\xi_1})y\\
\leq&[C'\delta_n\{s_n\log(p_n/s_n)\}^{1/2}+C''\delta_n^2(1+\delta_n) n^{1/2}\|\beta^{(n)}_{\xi^*\backslash\xi_1}\| \\
&+(1+C'\delta_n)\{1+\|\beta_{\xi^*\backslash\xi_1}^{(n)}\|^2\}^{1/2}\{(2+c_4)|\xi_2|\log(p/|\xi_2|)\}^{1/2}]^2\\
\leq&\left[C_0 \{s_n^2\log^2(p_n/s_n)/n\}^{1/2}+\{(2+\epsilon)|\xi_2|\log(p_n/|\xi_2|)\}^{1/2}\right]^2
 \end{split}
\]
for some constant $C_0$, where $c_4>2c_3$ can be arbitrarily small if we choose $c_3$ to be sufficiently small.

\end{proof}

{\noindent \bf Proof of Theorem 4.2}
\begin{proof}
We only prove the case that $s_n\prec p_n$. The results for general case that $\limsup s_n/ p_n <\zeta$ can be proved in a similar way.

By Corollary 5.35 of \cite{Vershynin2012} and $s_n\log(p_n/s_n)\prec n$, condition C2 holds 
with dominating probability,
thus by Lemma \ref{SCthm}, $|\widehat\xi|<Cs$ and $\|\beta_{\xi^*\backslash\widehat\xi}^{(n)}\|\leq M\{s_n
\log(p_n/s_n)/n\}^{1/2}$ for some $C$ and $M$, where $\widehat\xi=\xi(\widehat\beta)$ and $\xi^*=\xi(\beta^{(n)})$.
Let $\xi_1=\xi^*\cap\widehat\xi$ and $\xi_2=\widehat\xi\backslash\xi_1$, then by Lemma \ref{diff},
with high probability, 
 \begin{equation}\label{l3}
\begin{split}
&\mbox{SC}(\xi_1)-\mbox{SC}(\xi_2\cup\xi_1)=y^T(P_{\xi_1\cup\xi_2}-P_{\xi_1})y-\gamma\sum_{i=|\xi_1|+1}^{|\xi_2\cup\xi_1|} \log (p/i)\\
\leq&\left[C_0 \{s_n^2\log^2(p_n/s_n)/n\}^{1/2}+\{(2+\epsilon)|\xi_2|\log(p_n/|\xi_2|)\}^{1/2}\right]^2 - \gamma\sum_{i=|\xi_1|+1}^{|\xi_2\cup\xi_1|} \log (p/i),
\end{split}
\end{equation}

Note that since $s\log(p/s)\prec n$, hence $ \{s^2\log^2(p/s)/n\}^{1/2}= o (\{s\log(p/s)\}^{1/2})$. Therefore, it is not difficult to see that (\ref{l3}) asymptotically is 
negative if $|\xi_2|\geq \delta 's$ for any positive $\delta'>0$, as long as $\gamma>(2+\epsilon)$.
If furthermore, $s\log(p/s)\prec {n}^{1/2}$, then $\{s^2\log^2(p/s)/n\}^{1/2}=o(1)$, and by the same argument used in the proof
of Theorem 4.1, we can show that (\ref{l3}) asymptotically is 
negative if $|\xi_2|\geq \max\{1, |\xi_1|(|\xi_1|/p)^{\gamma'}\}$ where $\gamma'$ satisfies $(1+\gamma')(2+\epsilon)<\gamma$.
This hence proves our claim on the number of false positive selection.

Now, we study the $L_2$ estimation error 
$\|\widehat\beta-\beta^{(n)}\|^2=\sum_{i\in\xi^*\backslash\widehat\xi}[\beta_i^{(n)}]^2+\sum_{i\in\widehat\xi}
(\widehat\beta_i-\beta^{(n)}_i)^2$. 
Recall that given a large $n$, with high probability, condition C2 holds with
$1-\delta<\underline\lambda<\overline\lambda<1+\delta$ for any fixed arbitrarily small constant $\delta$. 
And our following analysis is conditional on this event.
First of all,
\begin{equation}\label{l4}\begin{split}
&[\sum_{i\in\widehat\xi}
(\widehat\beta_i-\beta^{(n)}_i)^2]^{1/2}
=\|(X_{\widehat\xi}^TX_{\widehat\xi})^{-1}X_{\widehat\xi}^T\varepsilon\|+\|(X_{\widehat\xi}^TX_{\widehat\xi})^{-1}X_{\widehat\xi}^TX_{\xi^*\backslash\widehat\xi}
\beta_{\xi^*\backslash\widehat\xi}^{(n)}\|\\
\leq& \sup_{\xi: |\xi\backslash\xi^*|\leq \delta's}\|(X_{\widehat\xi}^TX_{\widehat\xi})^{-1}X_{\widehat\xi}^T\varepsilon\|+ \{(2\delta/(1-\delta))\sum_{i\in\xi^*\backslash\widehat\xi}[\beta_i^{(n)}]^2\}^{1/2}\\
\leq& \delta'O(s_n\log(p_n/s_n))^{1/2}+ \{(2\delta/(1-\delta))\sum_{i\in\xi^*\backslash\widehat\xi}[\beta_i^{(n)}]^2\}^{1/2},
\end{split}
\end{equation}
with high probability.
Therefore, it suffices to study the value $\|\beta_{\xi^*\backslash\widehat\xi}^{(n)}\|$.

Since $\widehat\xi$ is the solution of minimization, 
\begin{equation}\label{l1}
\begin{split}
0&>\mbox{SC}(\widehat\xi)-\mbox{SC}(\widehat\xi\cup\xi^*)\geq \mbox{RSS}(\widehat\xi)-\mbox{RSS}(\widehat\xi\cup\xi^*)-\gamma\sum_{i=1}^{s_n} \log (p_n/i)\\
&=\|(P_{\widehat\xi\cup\xi^*}-P_{\widehat\xi})X\beta^{(n)}+(P_{\widehat\xi\cup\xi^*}-P_{\widehat\xi})\varepsilon)\|^2-\gamma\sum_{i=1}^{s_n} \log (p_n/i)\\
&\geq\left(\{(1-\delta)\sum_{i\in\xi^*\backslash\widehat\xi}[\beta_i^{(n)}]^2\}^{1/2}-\|(P_{\widehat\xi\cup\xi^*}-P_{\widehat\xi})\varepsilon\|\right)^2-\gamma\sum_{i=1}^{s_n} \log (p_n/i).
\end{split}
\end{equation}
Note that $(P_{\widehat\xi\cup\xi^*}-P_{\widehat\xi})\varepsilon = (I-P_{\widehat\xi})X_{\widetilde\xi}(X_{\widetilde\xi}^T(I-P_{\widehat\xi})X_{\widetilde\xi})^{-1}X_{\widetilde\xi}^T(I-P_{\widehat\xi})\varepsilon$,
 where $\widetilde\xi =\xi^*\backslash\widehat\xi$, hence 
$\|(P_{\widehat\xi\cup\xi^*}-P_{\widehat\xi})\varepsilon\|\leq \{1/n(1-\delta)\}^{1/2}\|X_{\widetilde\xi}^T(I-P_{\widehat\xi})\varepsilon\|
\leq \{1/n(1-\delta)\}^{1/2}(\|X_{\widetilde\xi}^T\varepsilon\|+\|X_{\widetilde\xi}^T P_{\widehat\xi}\varepsilon\|).$
Note that $|\widehat\xi|\leq (1+\delta')s_n$, and by Lemma \ref{chi}, for any fixed $\delta''>0$, we have
\begin{equation}\label{l2}
\begin{split}
&\sup_{|\xi|\leq (1+\delta')s_n}\|X_{\xi^*\backslash\xi}^T\varepsilon\|+\|X_{\xi^*\backslash\xi}^T P_{\xi}\varepsilon\|\\
\leq& \{n(1+\delta)\}^{1/2}\{(s+2\log 2^s+(\delta''/3) s\log(p/s))\}^{1/2}\\
&+\{2n\delta\}^{1/2}\left[(1+\delta')s+(2+\delta''_1/3)\log{p\choose (1+\delta')s}\right]^{1/2}
\end{split}
\end{equation}
holds with probability at least $1-\exp\{-cs\log(p/s)\}$ for some $c>0$.

Combine (\ref{l1}), (\ref{l2}) and (\ref{l4}), it is easy to see that with high probability,
$n\|\widehat\beta-\beta^{(n)}\|^2\leq (\gamma+c)s_n\log (p_n/s_n)$ for any fixed $c>0$,
as long as we choose $\delta$, $\delta'$ and $\delta''$ to be sufficiently small.

\end{proof}

\section{Incompatibility between rate minimaxity and selection consistency under regression models}

 In section 3, we introduce the concept of monotone estimator under normal means models, and establish the incompatibility between rate minimaxity and selection consistency. However, as mentioned, the concept of monotonicity doesn't naturally carry over to general linear regression models unless the design matrix is exact orthogonal. 
 By random matrix theory, e.g., \cite{Vershynin2012}, under condition (C1), with high probability, the singular values of low dimensional submatrix of $X$ are very close to 1, i.e., the columns in  $X$ are nearly orthogonal to each other. Hence, we conjecture that Lemma 3.1 still holds for regression models under condition (C1). That is, a selection-consistent estimator $\widehat\beta(X,y)$, which possesses monotonicity under normal orthogonal models (i.e., $\widehat\beta(X,y)$ is monotone with respect to $X^Ty$ when $X$ is exactly orthogonal), still ensures that there is no false discovery. Further theoretical investigation on this matter is beyond the scope of this work. 
 But we indeed can show that  a selection-consistent penalized estimator, under proper conditions, ensures that there is no false discovery. Note that penalized estimator, under exact orthogonal design, is usually monotone.

Let $\widehat\beta(X,y)$ be the penalized estimator using penalty function $pe(\beta)=\sum_jp(\beta_j, p_n, n)$, i.e.,
$\widehat\beta = \underset{\beta}{\mathrm{argmin}}\|y-X\beta\|^2/\sigma^2+pe(\beta)$, where $\sigma^2$ is the known variance for random error term.
Assume the function $p$ satisfies
\begin{itemize}
 \item[(D1):] $\arg\min_{\beta_j}[(y_j-\beta_j)^2+\sigma^2p(\beta_j)]\neq 0$ if and only if $|y_j|> \sigma h(p_n,n)$,
\end{itemize}
where $h(p_n,n)=h$ is the thresholding value caused by the penalization under exact orthogonal design.
Furthermore, we assume that under random Gaussian design, the estimator has the following properties:
\begin{itemize}
 \item[(D2):] Its convergence rate is $r_n=r_n(n,p_n,s_n)$, i.e., $\lim_nP_{\beta^{(n)}}(\|\widehat\beta(X,y)-\beta^{(n)}\|\leq r_n) =1$;
 \item[(D3):] $\lim_nP_{\beta^{(n)}}(\|\widehat\beta(X,y)\|_0\leq C_0s_n) =1$ for some $C_0>0$;
 \item[(D4):] $\lim_{n}P_{\beta^{(n)}}\{\xi(\widehat\beta(y))=\xi(\beta^{(n)})\}=1$, if $\|\beta^{(n)}\|_{\min}\geq t(n,s_n,p_n)$ for some function $t$,
\end{itemize}
for any sequence of $\beta^{(n)}\in \mB(p_n,s_n)$ satisfying $(s_n,p_n)_{n=1}^\infty\in\mathcal D$ where $\mathcal D$ is some set of configurations of 
dimension and sparsity growth.
Then, we have that
\begin{theorem}\label{prop}
 Any estimator $\widehat\beta$ that meets the above conditions must satisfy
 \[\lim_nP_{\beta^{(n)}}(\mbox{FP}(\widehat\beta)=0) =1,\]
 for any sequence of $\beta^{(n)}\in \mB(p_n,s_n)$ satisfying $(s_n,p_n)_{n=1}^\infty\in\mathcal (D\cap \widetilde{\mathcal D})$
  where $\widetilde{\mathcal D}=\{(s_n, p_n)_{n=1}^\infty:\limsup s_n/p_n\leq c< 1, \lim p_n=\infty \mbox{ and } s_n^{1/2}\log p_nr_n\prec 1\}$.
\end{theorem}
\begin{proof}
If this theorem is not true, then there exists a sequence of 
$\beta^{(n)}\in \mB(p_n,s_n)$ satisfying $(s_n,p_n)_{n=1}^\infty\in\mathcal D\cap\widetilde{\mathcal D}$, such that
$P_{\beta^{(n)}}(\mbox{FP}(\widehat\beta(X,y))\geq 1) \geq C>0$ for all $n$ where $C$ is some positive constant. By conditions D2 and D3,
$P_{\beta^{(n)}}(\mbox{FP}(\widehat\beta)\geq 1, \|\widehat\beta\|_0\leq C_0s_n\mbox{ and } \|\widehat\beta-\beta^{(n)}\|\leq r_n) \geq C'$ for some $C'\in(0,C)$.

Note that $\widehat\beta_i\neq0$ implies that
\[
\begin{split}
&\underset{\beta_i}{\mathrm{argmin}}\sum_{k=1}^n(y_k-\sum_{j\neq i}^{p_n}x_{kj}\widehat\beta_j-x_{ki}\beta_i)^2+p(\beta_i)\\
=&\underset{\beta_i}{\mathrm{argmin}}\left(\beta_i-\frac{\sum_{k=1}^n x_{ki}(\varepsilon_k+\sum_{j\neq i}^{p_n}x_{kj}(\beta_j^{(n)}-\widehat\beta_j)) }{\sum_{k=1}^nx_{ki}^2}\right)^2+\frac{p(\beta_i)}{\sum_{k=1}^nx_{ki}^2}
\neq 0
\end{split}
\]

Therefore, by condition D1, with probability as least $C'$, $\|\widehat\beta\|_0\leq C_0s_n$, $\|\widehat\beta-\beta^{(n)}\|\leq r_n$ and 
\[
\max_{i\notin\xi(\beta^{(n)})}\bigg|\frac{\sum_{k=1}^n x_{ki}(\varepsilon_k+\sum_{j\neq i}^{p_n}x_{kj}(\beta_j^{(n)}-\widehat\beta_j)) }{(\sum_{k=1}^nx_{ki}^2)^{1/2}}\bigg|\geq h(p_n,n).
\]
Combining with fact that with dominating probability, the singular values of any $n$ by $(C_0)s_n$ submatrix of $X$ is  bounded by $ n^{1/2}(1+O(s_n\log p_n/n)^{1/2}) $,
the above inequality implies that
\[
 \max_{i}\bigg|\frac{\sum_{k=1}^{n}\varepsilon_kx_{ki}}{(\sum_{k=1}^nx_{ki}^2)^{1/2}}\bigg|\geq h(p_n,n)-O(r_n(s_n\log p_n)^{1/2}), \mbox{ with probability at least } C''>0,
\]
where $0<C''<C'$.

Note that $Z_i=\sum_{k=1}^{n}\varepsilon_kx_{ki}/(\sum_{k=1}^nx_{ki}^2)^{1/2}\sim N(0,1)$, thus
\begin{equation}\label{pr}
\begin{split}
&P\left(\max_{i}\bigg|\frac{\sum_{k=1}^{n}\varepsilon_kx_{ki}}{(\sum_{k=1}^nx_{ki}^2)^{1/2}}\bigg|\geq h(p_n,n)\right) \geq 1-(2\Phi[\Phi^{-1}((1-C''/2)^{1/p_n})+O(r_n(s_n\log p_n)^{1/2})]-1)^{p_n}\\
\geq& 1-(2\Phi[\Phi^{-1}(\exp(-C''/2p_n))+O(r_n(s_n\log p_n)^{1/2})]-1)^{p_n}\\
\geq& 1-\{2\exp(-C''/2p_n) -1 -2\phi[\Phi^{-1}(\exp(-C''/2p_n))]O(r_n(s_n\log p_n)^{1/2})\}^{p_n}.
\end{split}\end{equation}

By the facts $\phi(x)\leq (1-\Phi(x))[(x^2+1)/x]$, $\Phi^{-1}(\exp(-C''/2p_n))\asymp (\log p_n)^{1/2}$, 
and the assumption $O(r_n(s_n\log p_n)^{1/2})=o(1/(\log p_n)^{1/2})$, we have
\[
\begin{split}&\phi[\Phi^{-1}(\exp(-C''/2p_n))]O(r_n(s_n\log p_n)^{1/2})=[1-\exp(-C''/2p_n)]O(\log p_n)^{1/2})
O(r_n(s_n\log p_n)^{1/2})\\
=&o([1-\exp(-C''/2p_n)])=o(1/p_n)
\end{split}\]

Therefore, (\ref{pr}) reduces to
\begin{equation}\label{pr2}
P\left(\max_{i}\bigg|\frac{\sum_{k=1}^{n}\varepsilon_kx_{ki}}{(\sum_{k=1}^nx_{ki}^2)^{1/2}}\bigg|\geq h(p_n,n)\right)\gtrsim 1-\exp\{-C''\}>0.
\end{equation}

Result (\ref{pr2}) implies that if the true parameter is a 0 vector, with probability at least $1-\exp\{-C''\}$, $\xi(\widehat\beta)\neq\emptyset$, and this contradict to
condition D4.
\end{proof}

\begin{remark}
Conditions D1-D4 hold for many popular choices of penalty functions, including
LASSO, SCAD and $L_1$ penalties with $r_n^2\asymp s_n\log p_n/n$, and under such cases, the set
$\widetilde{\mathcal D}=\{(s_n, p_n)_{n=1}^\infty: \limsup s_n/p_n\leq c< 1, \lim p_n=\infty \mbox{ and } s_n^2\log^3p_n\prec n\}$.
\end{remark}
 
 \section{FDR convergence under sharp minimaxity}
 Our result in the main manuscript reveals that under $s_n\prec p_n$, the FDR decreases polynomially for rate-minimax estimators $\widehat\beta$'s with suboptimal multiplicative constant, i.e., $\lim\sup_{\beta\in\mB(p_n,s_n)}E\|\beta-\widehat\beta\|^2/R_{opt}=c>1$. On the other hand, the false discovery control behavior of sharply minimax estimators (i.e.,  $\widehat\beta$ satisfying $\lim\sup_{\beta\in\mB(n,s_n)}E_\beta\|\beta-\widehat\beta\|^2/R_{opt}=1$) is still unclear. Therefore, we present one toy simulation, aiming to understand the decay of FDR empirically under sharp minimaxity. Similarly to the example demonstrated in the Introduction section, we choose a normal means model, with  $p_n=n$, $X=I$ and $s_n=n^{1/2}$ and all nonzero $\beta_i$'s are set as $[2\log(n/s)]^{1/2}$, and consider the estimator (4.1) with tuning parameter value $\gamma=2.0$. Note that by Theorem 2 of \cite{WuZ2013}, this estimator is sharply minimax. Figure \ref{fdpplot2} plots the logarithm of estimated FDR based on 100 independent simulations versus the logarithm of 
 true sparsity ratio $\log(s_n/n)$, where $n$ ranges form $2^8*100$ to $2^{15}*100$. We still observe that the FDR decays to zero as sparsity level decreases. Despite the nearly linear pattern shown by the top plot in Figure \ref{fdpplot2}, the U-shaped residual plot suggests that the decay of $\mbox{FDR}$ is sub-polynomial, w.r.t. the sparsity level.

\begin{figure}[htp]
 \begin{center}
  \includegraphics[width=10cm]{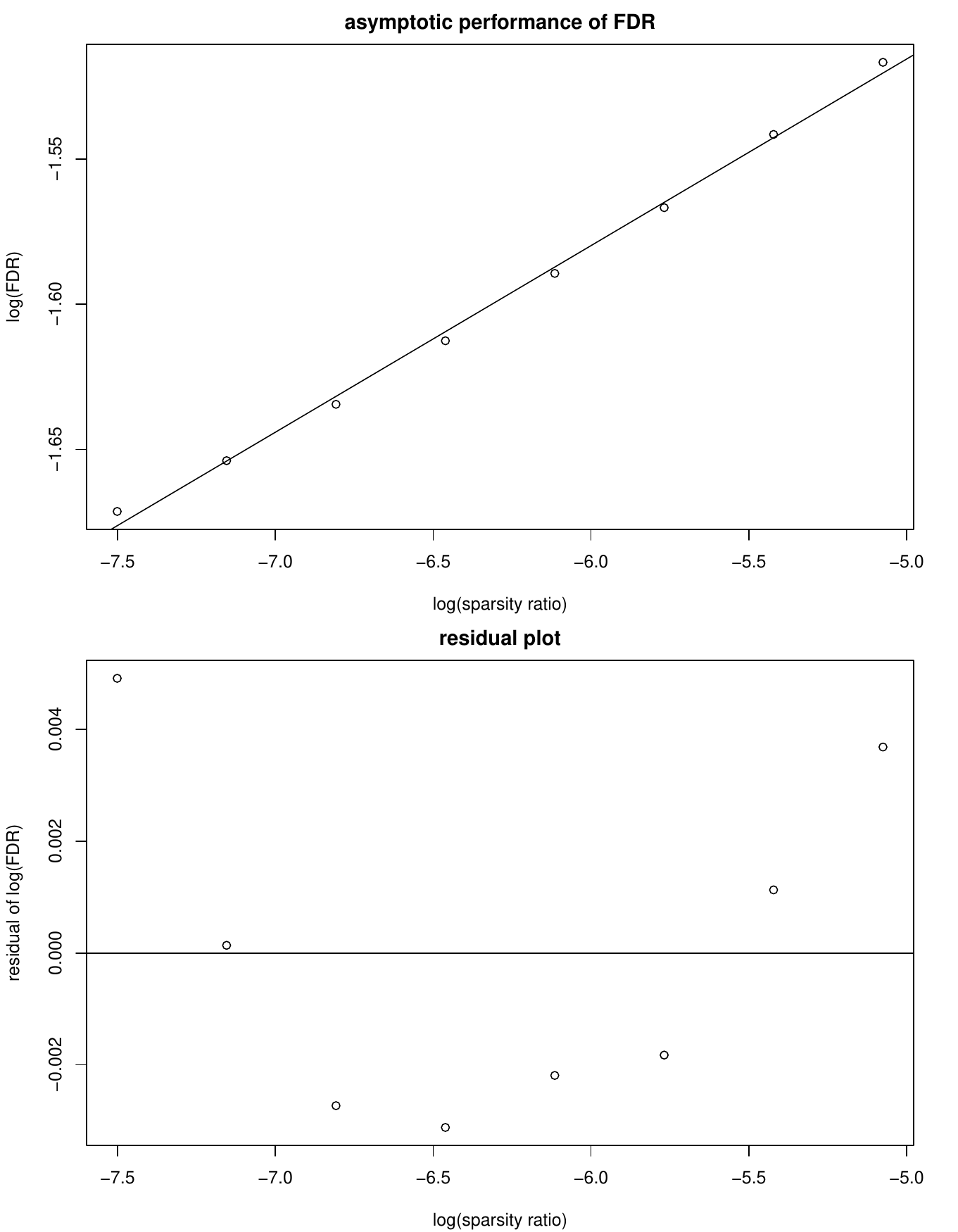}\label{fdpplot2}
  \caption{This figure plots the patter of FDR with respect to the sparsity ratio, for a sharply minimax estimator.  }
 \end{center}
\end{figure}

\vskip 0.2in
\bibliographystyle{plain}
\bibliography{ref}